\documentclass{article}
\usepackage{graphicx}
\usepackage{amsmath}
\usepackage{amsthm}
\usepackage{float}
\usepackage{graphicx}
\usepackage{array}
\usepackage{caption}
\usepackage{tikz-cd}
\usepackage{changepage}
\usepackage{ytableau}
\usepackage{amsthm}
\usepackage{thmtools}
\usepackage[round,authoryear]{natbib}
\usepackage{hyperref}
\usepackage{amssymb}
\usepackage{listings}
\setcitestyle{authoryear,round}
\title{The Hesse Pencil Variety}
\author{Elisabetta Rocchi\\
\\
Sorbonne Université, CNRS, LIP6, F-75005 Paris, France\\
\texttt{elisabetta.rocchi@sorbonne-universite.fr}}
\date{}
\usepackage[a4paper,margin=3.5cm]{geometry}

\newtheorem{teo}{Theorem}[section]
\newtheorem{lem}[teo]{Lemma}

\newtheorem{prop}[teo]{Proposition}
\newtheorem{defn}[teo]{Definition}

\newtheorem{oss}[teo]{Remark}

\newcommand{\revised}[1]{{\color{black}#1}}

\begin{document}

\maketitle
\tableofcontents
\begin{abstract}
    We introduce and study the Hesse pencil variety $H_8$, 
    obtained as the Zariski closure in the Grassmannian $G(1,9)$ of 
    the set of pencils generated by a smooth plane cubic and its Hessian. 
    We prove that $H_8$ has dimension $8$ and can be realized as the intersection 
    of $G(1,9)$ with ten hyperplanes corresponding to the Schur module $\mathbb{S}_{(5,1)}\mathbb{C}^3$. 
    Moreover, $H_8$ coincides with the closure of the special linear group $SL(3)$-orbit 
    of the pencil $\langle x^3+y^3+z^3,\ xyz\rangle$ and contains eight additional orbits. 
    The variety is singular, and its singular locus is precisely the union of two orbits, 
    $O(\langle x^3,x^2y\rangle)$ and $O(\langle x^2y,x^2z\rangle)$.

A key ingredient in our study is a cubic skew-invariant 
$R \in \bigwedge^3(\mathrm{Sym}^3\mathbb{C}^3)$, defined by 
$R(l^3,m^3,n^3) = (l \wedge m \wedge n)^3$, where $l,m,n$ are linear forms in $(\mathbb{C}^3)^\ast$. 
The vanishing of $R$ characterizes pencils generated by a cubic and its Hessian, 
and it allows us to write explicit equations defining $H_8$. 
A crucial geometric step in our argument is the fact that through four general points of 
$\mathbb{P}^2$ there pass exactly six Hesse configurations, which enables us to compute 
the multidegree of $H_8$ and conclude that it coincides with the variety defined by the invariant $R$.

\end{abstract}

\bigskip
\noindent\textbf{Keywords:} Algebraic geometry, Plane cubic curves, 
Grassmannians, Hesse configuration, Inflection points.

\section{Introduction}
\revised{We denote by $\mathbb{P}^m$ the projective space of dimension $m$ over the complex numbers.}
Let $d\geq 3 $ and $H_{(d,n)}: \mathbb
{P}^{\binom{n+d}{d}-1}\dashrightarrow \mathbb{P}^{\binom{n+\overline{d}}{\overline{d}}-1}$ 
be the Hessian map, which sends a degree $d$ forms to a form of degree $\overline{d}=(d-2)(n+1)$. 
Unless $(d,n)=(3,1)$ we have always $d\leq \overline{d}$ and equality holds only for binary quartics 
$(d,n)=(4,1)$ and plane cubics $(d,n)=(3,2)$. In these cases, we obtain the following maps, 
\begin{equation*}
    H_{(4,1)}:\mathbb{P}^4 \dashrightarrow \mathbb{P}^4 \ \ \ \ \ \ \ \ \ \ \ \ \ \ \ \ \ H_{(3,2)}: 
    \mathbb{P}^9 \dashrightarrow \mathbb{P}^9
\end{equation*}
Given $f\in \mathbb{P}(\mathrm{Sym}^4 \mathbb{C}^2)$ or $\mathbb{P}(\mathrm{Sym}^3\mathbb{C}^3)$, 
let $H(f)$ be its image under the Hessian map. If $f$ is smooth, we can consider the pencil 
\begin{equation*}
    \mu f+\lambda H(f), \ \ \ \ \mu,\lambda \in \mathbb{C}
\end{equation*}
which we denote by $\langle f, H(f)\rangle$, and refer to as a Hesse Pencil. \\

In the case of binary quartics, the Hesse pencils correspond to points of $G(1,4)$, 
the Grassmannian of lines in $\mathbb{P}^4$.
\begin{restatable}{defn}{defX}\label{defX}[The Hesse Pencil Variety of binary quartics]
    Consider the set of Hesse pencils.
    We denote by $H_3$ the subvariety of $G(1,4)$ obtained as the Zariski closure of this space, 
    that is 
    \begin{equation*}
        H_3=\overline{\{\langle f,H(f)\rangle \in G(1,4)\  | f\ \ has\ \ 4\ \ simple\ \ roots \}}
    \end{equation*}
\end{restatable}

Decomposing the space $\bigwedge^2(\mathrm{Sym}^4\mathbb{C}^2)$ into Schur modules under 
the action of $SL(2)$, we find:
\begin{equation*}
    \bigwedge\nolimits^2(\mathrm{Sym}^4\mathbb{C}^2)\cong \mathrm{Sym}^6\mathbb{C}^2\oplus \mathrm{Sym}^2\mathbb{C}^2.
\end{equation*}
The variety $H_3$ is the intersection $\mathbb{P}\left(\widehat{G(1,4)}\cap \mathrm{Sym}^6\mathbb{C}^2\right)$, 
where $\widehat{G(1,4)}$ is the affine cone over $G(1,4)$. It is a smooth Fano threefold, 
where the subscript 3 indicates the dimension. Moreover, $H_3$ coincides with the closure 
of the orbit of the pencil $\langle x^4+y^4,x^2y^2\rangle$, and contains, besides this orbit, 
two additional orbits \revised{(see Table~\ref{3orbitquartics})}. \\

For plane cubics, the Hesse pencils correspond to points of $G(1,9)$, the Grassmannian 
of lines in $\mathbb{P}^9$.
\begin{restatable}{defn}{defXX}\label{defXX}[The Hesse Pencil Variety of plane cubics]
    Consider the set of Hesse pencils.
    We denote by $H_8$ the subvariety of $G(1,9)$ obtained as the Zariski closure of 
    this space, that is 
    \begin{equation*}
        H_8=\overline{\{\langle f,H(f)\rangle \in G(1,9)\  | f\ \ is\ \ smooth\}}
    \end{equation*}
\end{restatable}

The main result we shall prove concerns this variety and can be stated as follows.
\begin{restatable}{teo}{mainthm}\label{teoS=N}
The Hesse pencil variety $H_8$ is the intersection of $G(1,9)$ with $10$ hyperplanes. 
It coincides with the closure of the $SL(3)$-orbit of the pencil 
\[
\langle x^3+y^3+z^3,\; xyz\rangle.
\]
In particular, it is irreducible and contains $8$ additional orbits of pencils. 
Moreover, $H_8$ is singular, and its singular locus is precisely the union of the two orbits 
\[
O(\langle x^3,x^2y\rangle) \quad \text{and} \quad O(\langle x^2y,x^2z\rangle).
\]
\end{restatable}
Let
\begin{equation*}
\bigwedge\nolimits^2(\mathrm{Sym}^3\mathbb{C}^3)\cong \mathbb{S}_{(5,1)}\mathbb{C}^3\ \oplus \ 
\mathbb{S}_{(3,3)}\mathbb{C}^3
\end{equation*}
be the Schur module decomposition with respect to the action of $SL(3)$. 
The variety $H_8$ is the intersection $\mathbb{P}\left(\widehat{G(1,9)}\cap 
\mathbb{S}_{(5,1)}\mathbb{C}^3\right)$. 
\revised{In particular, the $10$ hyperplanes appearing in Theorem~\ref{teoS=N} 
are precisely those defining $\mathbb{S}_{(5,1)}\mathbb{C}^3$.} \\
To study this variety, we introduce a skew invariant of plane cubics, 
$R\in \bigwedge^3(\mathrm{Sym^3\mathbb{C}^3})$, 
\begin{equation*}
    R: \mathrm{Sym}^3\mathbb{C}^\revised{3} \times \mathrm{Sym}^3\mathbb{C}^\revised{3}\times 
    \mathrm{Sym}^3\mathbb{C}^\revised{3} \longrightarrow\mathbb{C}
\end{equation*}
such that $R(l^3,m^3,n^3)=(l\wedge m\wedge n)^3$, where $l,m$ and $n$ denote three linear forms. 
We show that this invariant satisfies the following property: 
\begin{restatable}{cor}{corfg}\label{R(f,H(f),-)=0}
     If $f$ and $g$ are two cubic forms belonging to the same pencil generated by a cubic form 
     and its Hessian, then $R(f,g,-)=0$.  
\end{restatable}
We define 
\begin{equation*}
    N:= \overline{\{\langle f,g\rangle\in G(1,9)\ |\ R(f,g,-)=0\}}\subset G(1,9)
\end{equation*}
which is obtained as the intersection of ten linear equations in the Pl\"ucker coordinates 
(arising from the explicit expression of the invariant $R$) with the Pl\"ucker quadrics that 
define the Grassmannian $G(1,9)$. \\
From Corollary \ref{R(f,H(f),-)=0}, it follows that $H_8\subset N$. The key step in the proof 
of Theorem \ref{teoS=N} is to establish the equality 
\begin{equation*}
    H_8=N.
\end{equation*}
We show that $N$ and $H_8$ have the same dimension, namely $8$, and the same multidegree with 
respect to the Schubert cycles of dimension $8$ in $G(1,9)$, namely $(1,3,9,12,6)$. 
We show that $N$ and $H_8$ have the same dimension, namely $8$, and the same multidegree 
with respect to the Schubert cycles of dimension $8$ in $G(1,9)$, namely $(1,3,9,12,6)$. 
\revised{Recall that, by the Basis Theorem (see \cite{Schubertcalculus}), the cohomology group 
$H^*(G(1,9),\mathbb{Z})$ is a free abelian group generated by the Schubert cycles.
In particular, the Schubert cycles of codimension $p$ form a basis of $H^{2p}(G(1,9))$, 
so the multidegree of a subvariety of $G(1,9)$ is uniquely determined by its 
intersection numbers with the Schubert varieties of complementary dimension.}
In particular, we will compute the dimension and the multidegree of $N$ computationally, 
since its defining equations are known, and derive them theoretically for $H_8$.\\
Recall that the inflection points of a smooth cubic curve form a well-defined configuration 
known as the \textit{Hesse configuration}. The following result will be crucial in the 
proof of the multidegree of $H_8$: 
\begin{restatable}{prop}{seiconfigurazioni}\label{configurazioni4punti}
    There are exactly six Hesse configurations passing through four general points. 
    \end{restatable}
From these results, it follows that $H_8$ and $N$ coincide in dimension $8$, but there 
may exist lower-dimensional subvarieties of $N$ not contained in $H_8$. To complete the 
proof of Theorem \ref{teoS=N}, we classify all orbits of pencil contained in $N$, and 
show that for each of them, one can find a degeneration of pencils in $H_8$ that has 
the limit to the given pencil. \\
Finally, to study the singular locus, it is enough to evaluate the rank of the Jacobian 
matrix obtained by the equations of $N$. \\

After this paper was written, I kindly received from Vladimir Popov two interesting 
papers~\cite{Pop1,Pop2}, where the 9-dimensional variety $X$ of flexes of plane cubics is 
constructed and studied. There exists a natural map $X \to H_8$ whose general fiber consists 
of the union of nine rational curves, relating $X$ to our variety $H_8$.

\section{Binary Quartics}
We denote by $[x,y]$ the homogeneous coordinates of $\mathbb{P}^1$. A quartic $f$ in 
this space is a polynomial
\begin{equation*}
    f=a_0x^4+4a_1x^3y+6a_2x^2y^2+4a_3xy^3+a_4y^4.
\end{equation*}
It can be identified with the coefficients $(a_0,..,a_4)$ of this polynomial and regarded 
as a point in $\mathbb{P}^4$. The Hessian of a quartic, named $H(f)$, is still a quartic 
whose coefficients are polynomial of degree $2$ in the $a_i$. \\
Let $Z$ denote the variety of quartic forms for which the Hessian \revised{map} is not defined. 
This variety has dimension $1$. 
Moreover, Z coincides with the variety of cones (see \cite{Coni}), that is, with the set of 
quartics defines by equations of the form:
\begin{equation*}
    f=(b_0x+b_1y)^4
\end{equation*}
$Z$ is called the $quartic \ rational  \ normal  \ curve$.\\
Let us consider quartic curves of the form 
\[
f=(c_0x^2+c_1xy+c_2y^2)^2,
\] 
that is, squares of quadratic forms. Denote by $Sq$ the Zariski closure of this set. 
Equivalently, $Sq$ can be characterized as the locus of quartics that coincide with their own Hessian. 
The variety $Sq$ is two-dimensional and contains the previously defined variety $Z$, i.e.\ $Z \subset Sq$. 
Moreover, it is cut out by seven cubic equations in the coefficients $a_i$, arising from the vanishing 
of the coefficients of
\[
\det\begin{pmatrix}
    f_x & f_y \\
    H(f)_x & H(f)_y
\end{pmatrix},
\]
which expresses the condition that $f$ and $H(f)$ are proportional, i.e. that the two quartics coincide 
in $\mathbb{P}^4$. \revised{Finally, $Sq$ is smooth: a computation in Macaulay2 
shows that the Jacobian matrix has maximal rank everywhere}.
\\

The special linear group $SL(2)$ acts on the space of quartic forms, 
and the Hessian map is equivariant with respect to this action (see \cite[Prop.~4.4.2]{Invariantstrumfels}), meaning that 
\begin{equation}
    H(C\cdot f)=C\cdot H(f)\ \ \ \ \ \ \forall \ C\in SL(2).
\end{equation}
 There are infinitely many \revised{orbits} containing quartics with four distinct roots, 
 and four additional 
 orbits containing singular quartics. \revised{The orbits are listed in Table~\ref{tab:my_label}, 
 adapted from \cite[p.~29]{Invariantiquartiche}. We denote by $O(f)$ the orbit of the quartic $f$.}

\begin{table}[H]
\begin{adjustwidth}{-2cm}{-2cm}
\centering
\renewcommand{\arraystretch}{1.5}
    \begin{tabular}{c|c|c|c}
    \textbf{Orbit representatives} & \textbf{Hessian} & \textbf{Description} 
    &\textbf{dim($\overline{O(f)}$)} \\ \hline 
     $x^4+6\lambda x^2y^2+y^4 \ \ \ \ \lambda\neq \pm \frac{1}{3}$& $x^4+y^4+6\frac{1-3\lambda^2}{6\lambda}x^2y^2\ \ \ \ 
     \lambda \neq 0$ & $simple\ \  roots$&3 \\
    \hline
      $x^4+x^2y^2$   & $6x^4-x^2y^2$& $one\ \ double\ \ root$&3\\ \hline
       $x^2y^2$  & $x^2y^2$ & $two\ \ double\ \ roots$&2 \\ \hline
       $x^3y$ &$x^4$ & $triple\ \ root$&2\\ \hline 
       $x^4$ & 0 & $quadruple\ \ root$&1
    \end{tabular}
    \caption{Orbits of binary quartics under the action of $SL(2)$.}
    \label{tab:my_label}
\end{adjustwidth}
\end{table}

\subsection{The Hesse Pencil Variety of Binary Quartics}
\begin{defn}[The Hesse Pencil]
    Let $f$ be a quartic with four distinct roots. The pencil generated by $f$ and its Hessian $H(f)$, 
    that is 
    \begin{equation*}
        \langle f,H(f)\rangle =\lambda f + \mu H(f)  \ \ \ \ \ \ \ \ \lambda, \mu \in \mathbb{C},
    \end{equation*} is called  $ Hesse\ pencil$. 
\end{defn}

\begin{prop}\label{pencilquadriche}
    Let $L$ be a Hesse pencil, $L=\langle f,H(f)\rangle$. Then, the Hessian of every quartic on the 
    line $L$ still lies on $L$. \revised{Moreover, if $q\in L\setminus Sq$, then every quartic $p$ with $H(p)=q$
    also lies in $L$.} 
\end{prop}
\begin{proof}
Every Hesse pencil is projectively equivalent to a pencil of the form
\[
x^4 + 6\lambda x^2y^2 + y^4, \quad \lambda \in \mathbb{C}\cup \{\infty\},
\]
where \(\lambda=\infty\) corresponds to \(x^2y^2\). A direct computation shows that the Hessian 
of any quartic in this pencil again belongs to the same pencil.  

Consider the map \(\lambda \mapsto \frac{1-3\lambda^2}{6\lambda}\) that sends a quartic to the 
parameter of its Hessian in the canonical form. This map is $2:1$, showing that a generic quartic 
in the pencil is the Hessian of exactly two other quartics in the same pencil. 
Moreover, the Hessian map has degree $2$ (see \revised{\cite[Theorem~4.1]{Hessianmap}}), which 
ensures that all \revised{preimages} of a generic quartic lie in the same pencil.  

The exceptions are the three quartics that coincide with their own Hessian, namely the points in $Sq$, 
where the fiber of the Hessian map is not finite. Solving \(6\lambda^2 = 1-3\lambda^2\) together 
with \(\lambda=\infty\) identifies precisely these three quartics.
\end{proof}

\defX*
As seen in the proof of Proposition \ref{pencilquadriche}, all Hesse pencils lie in the orbit of 
$\langle x^4+y^4,x^2y^2\rangle$, and therefore we have
\begin{equation*}
    H_3=\overline{O(\langle x^4+y^4,x^2y^2\rangle)}
\end{equation*}
\begin{prop}
    The Hesse pencil variety of binary quartics $H_3$ has dimension $3$. 
\end{prop}
\begin{proof}
    Consider 
    \[
        P:=\{(f,L)\in \mathbb{P}^4\times H_3 \mid f\in L\}\subset \mathbb{P}^4 \times G(1,4),
    \]
    with projections $p_1:P\to \mathbb{P}^4$ and $p_2:P\to H_3$.  
    For a generic $f\in \mathbb{P}^4$, the fiber $p_1^{-1}(f)$ is finite (in fact, it consists of 
    a unique $L=\langle f,H(f)\rangle$), hence $\dim(P)=4$.  
    On the other hand, for $L\in H_3$, the fiber $p_2^{-1}(L)$ is the line $L$ itself, so $\dim(p_2^{-1}(L))=1$.  
    It follows that
    \[
        \dim(P)=\dim(H_3)+1=4 \quad\Rightarrow\quad \dim(H_3)=3. 
    \]
\end{proof}

To derive the equations defining $H_3$ in Pl\"{u}cker coordinates, consider the $2\times5$ matrix 
whose first row consists of the coefficients $a_i$ of a generic quadric, and whose second row 
consists of the corresponding coefficients of its Hessian, expressed in terms of the $a_i$: 
\[
A=
\begin{bmatrix}
    a_0& a_1&a_2&a_3&a_4\\
    -6a_1^2+6a_0a_2& -3a_1a_2+3a_0a_3& -3a_2^2+2a_1a_3+a_0a_4& -3a_2a_3+3a_1a_4& -6a_3^2+6a_2a_4
 \end{bmatrix}.
 \]
The Pl\"{u}cker coordinates of the line defined by $f$ and $H(f)$ are then given by the $2\times2$ 
minors of this matrix, namely
\[
    p_{(i,j)}=\det(A_i|A_j), \qquad 0\leq i<j\leq 4,
\]
where $A_i$ denotes the $i$-th column of $A$. Hence the Pl\"{u}cker coordinates are cubic expressions in the $a_i$.  
Eliminating the variables $a_i$ yields an ideal containing the relations among the $p_{(i,j)}$.  \\
This computation was performed using \texttt{Macaulay2} (see \cite{M2}). The result shows that $H_3$ 
is defined by eight equations: 
three linear and five quadratic. The quadratic ones coincide with the Plücker relations defining the 
Grassmannian $G(1,4)$. 
Thus \revised{$H_3$} is the intersection of $G(1,4)\subset \mathbb{P}^9$ with three hyperplanes. 
In particular, the linear equations defining $H_3$ are
\begin{equation*}
    3p_{(2,3)}-p_{(1,4)}, \quad 
    2p_{(1,3)}-p_{(0,4)}, \quad 
    3p_{(1,2)}-p_{(0,3)}.
\end{equation*}

Let us now consider the vector space $\bigwedge^2 Sym^4\mathbb{C}^2$, which has dimension \revised{$10$}. 
We can naturally view the Grassmannian $G(1,4)$ as a subvariety of $\mathbb{P} 
\left( \bigwedge^2 Sym^4\mathbb{C}^2\right)$. 
The latter space decomposes into Schur modules as follows: 
\begin{equation}\label{decomposition}
\bigwedge\nolimits^2(Sym^4\mathbb{C}^2)\cong\mathbb{S}_{(7,1)}\mathbb{C}^2\ \oplus \ \mathbb{S}_{(5,3)}\mathbb{C}^2.
\end{equation}
Since the representations of $SL(2)$ are isomorphic to their duals, we can also write
\begin{equation*}
    \bigwedge\nolimits^2(Sym^4\mathbb{C}^2)\cong Sym^6\mathbb{C}^2\  \oplus\ Sym^2 \mathbb{C}^2.
\end{equation*} 

\begin{prop}\label{propA-F}
    With the above notation, we have 
    \begin{equation}
        \mathbb{P}\left(\widehat{G(1,4)}\cap Sym^6\mathbb{C}^2\right)=H_3
    \end{equation}
    where $\widehat{G(1,4)}$ denotes the affine cone over $G(1,4)$. 
\end{prop}

\begin{proof}
The variety $H_3$ is invariant under the action of $SL(2)$. The three linear equations defining 
$H_3$ span a $3$–dimensional $SL(2)$–stable subspace. By the decomposition in \ref{decomposition}, 
$\mathbb{S}_{(5,3)}\mathbb{C}^2$ is the unique $3$–dimensional $SL(2)$–invariant subspace.

Hence, the vanishing of these three linear forms defines the dual subspace, i.e., it projects onto 
the complementary summand $Sym^6\mathbb{C}^2$. Intersecting the affine cone $\widehat{G(1,4)}$ with 
this subspace and projectivizing gives precisely $H_3$, as claimed.
\end{proof}

From Proposition~\ref{propA-F} it follows that the variety $H_3$ admits a natural embedding in the 
space of binary sextics, providing an explicit correspondence between pencils of quartics and binary 
sextics. In particular, the pencil $\langle x^4+y^4, \, x^2y^2\rangle$ corresponds to the sextic $x^5y-xy^5$. 
This correspondence can be obtained explicitly using the decomposition of \(\bigwedge^2 Sym^4 \mathbb{C}^2\) 
into irreducible $SL(2)$--modules. 

The variety $H_3$ has been previously studied in \cite{Aluffi-Faber}. The results presented in this section 
therefore recover classical constructions, offering an alternative perspective on well-known facts.

\begin{oss}\label{canonicalbundlequartiche}
    We have seen that the variety $H_3$ is given by the intersection of the Grassmannian $G(1,4)$ with 
    three hyperplanes. According to \cite{Mukai}, it is a Fano 3-fold. A Fano variety $X$ is characterized 
    by having ample anticanonical bundle. In particular, for the Grassmannian $G(1,4)$, the canonical 
    bundle satisfies $K_{G(1,4)}\cong O(-5)$. By intersecting with three hyperplanes, we obtain a 3-fold 
    $H_3$ with $K_{H_3}\cong O(-2)$.
\end{oss} 

We now apply Schubert calculus to compute the multidegree of $H_3$ with respect to Schubert cycles.  
In $G(1,4)$ there are nine Schubert cycles generating $H^*(G(1,4),\mathbb{Z})$. Each of them 
corresponds to a Young tableau that can be embedded in the representative $2\times 5$ matrix 
\revised{(see \cite{FultonHook})}. 
Moreover, for a partition $\lambda$, one has $\mathrm{codim}(\Omega(\lambda))=\#\{\text{boxes in }\lambda\}$.\\
Since $H_3$ has dimension $3$, the relevant part of the cohomology is $H^6(G(1,4),\mathbb{Z})$. 
By the Basis Theorem in \cite{Schubertcalculus}, this group has two generators, represented by the tableaux
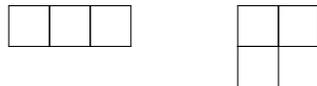
\begin{align*}
    \begin{ytableau}
      \empty  &\empty  & \empty
    \end{ytableau}
    &\qquad\qquad
    \begin{ytableau}
        \empty&\empty \\
        \empty\\
    \end{ytableau}
\end{align*} 
corresponding to the Schubert varieties $\Omega(0,4)$ and $\Omega(1,3)$, both of dimension $3$.\\
It follows that the degree of $H_3$ can be written as
\begin{equation}\label{multidegr}
    \deg(H_3)=\alpha\ \deg(X_{\scalebox{0.3}{\begin{ytableau}
        \empty&\empty &\empty \\
    \end{ytableau}}})  + \beta\ \deg(X_{\scalebox{0.3}{\begin{ytableau}
        \empty & \empty \\
        \empty \\
    \end{ytableau}}}),
\end{equation}
where the coefficients $\alpha$ and $\beta$ are given by the intersection numbers
\[
\alpha=\deg(H_3\cap X_{\scalebox{0.3}{\begin{ytableau}
        \empty&\empty &\empty \\
    \end{ytableau}}}), \qquad
\beta=\deg(H_3\cap X_{\scalebox{0.3}{\begin{ytableau}
        \empty &\empty \\
        \empty\\
    \end{ytableau}}}),
\]
since $\Omega(0,4)$ and $\Omega(1,3)$ are self-dual.  
These degrees can be computed using the hook-length formula in \cite{FultonHook}:
\[
    \deg\!\left(X_{\scalebox{0.3}{\begin{ytableau}
        \empty&\empty &\empty \\
    \end{ytableau}}}\right)=\frac{3!}{3\cdot 2\cdot 1}=1,
    \qquad
    \deg\!\left(X_{\scalebox{0.3}{\begin{ytableau}
        \empty &\empty \\
        \empty\\
    \end{ytableau}}}\right)=\frac{3!}{3\cdot 1 \cdot 1}=2.
\]

\begin{prop}
With the above notation, the multidegree of the variety $H_3$ is
\[
    \alpha=1, \qquad \beta=2.
\] 
In particular, this is compatible with the known fact that the degree of $H_3$ is $5$.
\end{prop}

\begin{proof}
Consider a Schubert variety corresponding to the cycle $\Omega(0,4)$, representing all lines 
in $\mathbb{P}^4$ passing through a fixed point. Intersecting with $H_3$, we count how many 
lines of the form $\langle f,H(f)\rangle \subset \mathbb{P}^4$ pass through a generic point $g$. 
By Proposition~\ref{pencilquadriche}, there is a unique such line, namely $\langle g,H(g)\rangle$, 
so we obtain $\alpha=1$.

Since the degree of $H_3$ is known to be $5$, the relation between the multidegree and the degree implies
\[
\deg(H_3) = \alpha \cdot 1 + \beta \cdot 2 = 5.
\]
From this equation we immediately deduce $\beta=2$.
\end{proof}

Since $H_3$ is $SL(2)$-invariant, it contains orbits of pencils. In particular, the following result holds.

\begin{teo}\label{teoSquartiche}
The variety $H_3$, defined in \ref{defX}, coincides with the closure of the orbit of the pencil 
$\langle x^4+y^4,x^2y^2\rangle$ under the action of $SL(2)$, that is
\[
    H_3=\overline{O(\langle x^4+y^4,x^2y^2\rangle)}.
\]
Moreover, $H_3$ consists of three orbits of pencils and is smooth.  
\end{teo}

\begin{proof}
For the proof of this theorem we refer to the article \cite{Aluffi-Faber}. Here, we simply 
report the correspondence between representatives of the $SL(2)$-orbits in the space of binary 
sextics and in the space of quartic pencils, together with the rank of the $8\times 10$ Jacobian 
matrix of the defining equations at points in each orbit. At smooth points, the Jacobian has rank 
equal to the codimension of $H_3$, namely $6$, so singularities would occur only where the rank drops.

The results are summarized in the following table:

\begin{table}[H]
\centering
\renewcommand{\arraystretch}{1.5}
\begin{tabular}{c|c|c|c}
    \textbf{Representative $\langle f,g\rangle $} & \textbf{Representative sextic}&
    \textbf{dim($O(\langle f,g\rangle )$)}& \textbf{Rank($J$)}\\
    \hline
      $\langle x^4+y^4,x^2y^2\rangle$ &$x^5y-xy^5$  & $3$ & $6$\\ \hline
      $\langle x^4,x^2y^2\rangle $ &$xy^5$& $2$ & $6$\\ \hline 
      $\langle x^4,x^3y\rangle$ &$x^6$& $1$ & $6$\\ \hline
\end{tabular}
\caption{Classification of the $SL(2)$-orbits in $H_3$.}
\label{3orbitquartics}

\end{table}
Since the Jacobian has full rank at all points, $H_3$ is smooth.
\end{proof}

\section{Plane Cubics}
We denote by $[x,y,z]$ the homogeneous coordinate of $\mathbb{P}^2$. A cubic 
$C_f\subset \mathbb{P}^2$ is defined by an equation of the form $f=0$ with  
\begin{equation*}
    f=a_0 x^3+3a_1 x^2y+3a_2 x^2z+3a_3xy^2+6a_4 xyz+3a_5 xz^2+a_6y^3+3a_7y^2z+3a_8yz^2+a_9z^3.
\end{equation*}
We can identify $C_f$ with $f$, and $f$ with the ten homogeneous coordinates given 
by its coefficients $a_i$. Therefore, we can think of a cubic curve as a point in $\mathbb{P}^9$. 
The Hessian of a plane cubic, named $H(f)$, is still a plane cubic whose coefficients 
are polynomial of degree $3$ in the $a_i$. \\
A smooth cubic curve contains nine inflection points, namely the points lying in the
intersection between the cubic itself and its Hessian. These points are arranged in a highly 
structured configuration (see \cite[Vol.~3, Ch.~3]{Enriquescubiche} for details). Any line passing 
through two of them always contains a third one; in total, there are twelve such lines, which are 
classically known as the Maclaurin lines. Such a configuration of $9$ points and $12$ lines in $\mathbb{P}^2$ 
is called the $Hesse \ \ configuration$. Moreover, if a smooth cubic curve $f$ passes through nine 
points arranged in the Hesse configuration, then these points are inflection points of $f$.
\begin{figure}[H]
\centering
    \begin{tikzpicture}

        \draw[domain=-3:3,smooth,variable=\x,black,thin] plot({\x},{0});
        \draw[domain=-3:3,smooth,variable=\x,black,thin] plot({\x},{-2});
        \draw[domain=-3:3,smooth,variable=\x,black,thin] plot({\x},{2}); 
        \draw[domain=-3:3,smooth,variable=\y,black,thin] plot({-2},{\y});  
        \draw[domain=-3:3,smooth,variable=\y,black,thin] plot({0},{\y});
        \draw[domain=-3:3,smooth,variable=\y,black,thin] plot({2},{\y});
       \draw[domain=-3:3,smooth,variable=\x,black,thin] plot({\x},{\x});  
       \draw[domain=-3:3,smooth,variable=\x,black,thin] plot({\x},{-\x});
       \fill[black] (-2,-2) circle (2pt) node[above left] {$p_1$};
        \fill[black] (-2,0) circle (2pt) node[above left] {$p_4$};
        \fill[black] (-2,2) circle (2pt) node[above left] {$p_7$};
       \fill[black] (0,-2) circle (2pt) node[above left] {$p_2$};
        \fill[black] (0,0) circle (2pt) node[above left] {$p_5$};
        \fill[black] (0,2) circle (2pt) node[above left] {$p_8$};
        \fill[black] (2,-2) circle (2pt) node[above left] {$p_3$};
        \fill[black] (2,0) circle (2pt) node[above left] {$p_6$};
        \fill[black] (2,2) circle (2pt) node[above left] {$p_9$};    

        \draw[orange, line width=1pt] (-0.5,2.5)--(2.5,-0.5) (-2.5,-1.5)--(-1.5,-2.5);
        \draw[pink, line width=1pt] (0.5,2.5)--(-2.5,-0.5) (1.5,-2.5)--(2.5,-1.5);
        \draw[purple, line width=1pt] (-0.5,-2.5)--(2.5,0.5) (-2.5,1.5)--(-1.5,2.5);
        \draw[green, line width=1pt] (0.5,-2.5)--(-2.5,0.5) (1.5,2.5)--(2.5,1.5);
     \end{tikzpicture}
     \caption{This figure represents nine points in Hesse configuration and the twelve lines 
     that characterize them. The eight black lines are clearly visible, while the remaining four 
     could not be explicitly drawn. Instead, four different colours were used to indicate the 
     three points on each of these lines. }
     \label{hesseconfiguration}
\end{figure}
We now turn to a result that will play a crucial role in the proof of Proposition \ref{multigradoScubiche}.
\seiconfigurazioni*

    \begin{proof}
Since any four points in general position in $\mathbb{P}^2$ can always be mapped,  via a projectivity, to 
\begin{equation*}
    (1:0:0)\ \ \ \ \ \ (0:1:0)\ \ \ \ \ \ (0:0:1)\ \ \ \ \ \ (1:1:1)
\end{equation*}
we may choose these four points as our fixed points.\\   
From now on, we will keep the convention of displaying the Hesse configuration as in Figure~\ref{hesseconfiguration}.
We first show that there are three possible choices for the fifth point. 
Then we show that, for each such choice, there are two possible Hesse configurations.
\revised{By symmetry, it suffices to treat only one of the three cases, since the other 
two are obtained by permuting the coordinates and yield the same number of solutions.}
\begin{itemize}
\item  \textbf{Choice of the fifth point}\\
We know that, in a Hesse configuration, any line passing through two points always contains a 
\revised{third point} of the configuration. The four fixed points determine six lines, so two of these lines 
must necessarily intersect at a fifth point of the configuration. Indeed, if this were not the case, 
we would have $4$ fixed points plus $6$ additional ones from the six lines, which would be too many 
to fit in the configuration. This gives us \revised{$\frac{1}{2}\binom{4}{2}=3$} ways to choose the fifth point, which correspond to:
\begin{figure}[H]
\centering
\begin{minipage}{0.45\textwidth}
    \begin{tikzpicture}
        \fill[blue] (-1,0) circle (1pt) node[above left] {$(0:0:1)$};
       \fill[blue] (0,-1) circle (1pt) node[below] {$(1:1:1)$};
        \fill[blue] (0,1) circle (1pt) node[above] {$(1:0:0)$};
        \fill[blue] (1,0) circle (1pt) node[above right] {$(0:1:0)$};
        \fill[black] (0,0) circle (1pt) node[above] {$(0:1:1)$};  
     \end{tikzpicture}
    \caption{$\{y=z\}\cap\{x=0\}= (0:1:1)$}
     \label{figuradue}
     \end{minipage} \hspace{0.5cm}
     \begin{minipage}{0.45\textwidth}
     \centering
     \begin{tikzpicture}
        \fill[blue] (-1,0) circle (1pt) node[above left] {$(0:0:1)$};
       \fill[blue] (0,-1) circle (1pt) node[below] {$(1:1:1)$};
        \fill[blue] (0,1) circle (1pt) node[above] {$(0:1:0)$};
        \fill[blue] (1,0) circle (1pt) node[above right] {$(1:0:0)$};
        \fill[black] (0,0) circle (1pt) node[above] {$(1:0:1)$};  
     \end{tikzpicture}      
     \caption{$\{x=z\}\cap\{y=0\}= (1:0:1)$}
     \end{minipage}
\end{figure}

\begin{figure}[H]
\centering
\begin{tikzpicture}
        \fill[blue] (-1,0) circle (1pt) node[above left] {$(0:1:0)$};
       \fill[blue] (0,-1) circle (1pt) node[below] {$(1:1:1)$};
        \fill[blue] (0,1) circle (1pt) node[above] {$(0:0:1)$};
        \fill[blue] (1,0) circle (1pt) node[above right] {$(1:0:0)$};
        \fill[black] (0,0) circle (1pt) node[above] {$(1:1:0)$};  
     \end{tikzpicture}
     \caption{$\{x=y\}\cap\{z=0\}=(1:1:0)$}
     \end{figure}
\item \textbf{Completation of the configuration (case of Figure~\ref{figuradue})}\\
Consider the line through $(1:0:0)$ and $(0:0:1)$, that is, $y=0$. On this line, there must be 
another point of the configuration, which we \revised{denote} as $(1:0:\lambda)$. This point generates three 
other lines of the configuration with the points $(0:1:0)$, $(0:1:1)$ and $(1:1:1)$, which are 
respectively $z=\lambda x$,  $z=y+\lambda x$ and $z=\lambda x +y(1-\lambda)$. On each of these lines, 
we find another point of the configuration. 

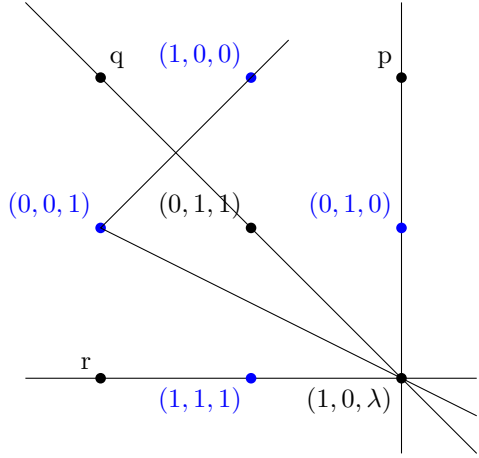
\begin{figure}[H]
\centering
    \begin{tikzpicture}
        \draw[domain=-3:3,smooth,variable=\y,black,thin] plot({2},{\y}); 
        \draw[domain=-3:3,smooth,variable=\x,black,thin] plot({\x},{-2});
       \draw[domain=-3:3,smooth,variable=\x,black,thin] plot({\x},{-\x});
       \fill[black] (-2,-2) circle (2pt) node[above left] {r};
        \fill[blue] (-2,0) circle (2pt) node[above left] {$(0:0:1)$};
        \fill[black] (-2,2) circle (2pt) node[above right] {q};
       \fill[blue] (0,-2) circle (2pt) node[below left] {$(1:1:1)$};
        \fill[black] (0,0) circle (2pt) node[above left] {$(0:1:1)$};
        \fill[blue] (0,2) circle (2pt) node[above left] {$(1:0:0)$};
        \fill[black] (2,-2) circle (2pt) node[below left] {$(1:0:\lambda)$};
        \fill[blue] (2,0) circle (2pt) node[above left] {$(0:1:0)$};
        \fill[black] (2,2) circle (2pt) node[above left] {p};    
        
        \draw[red] (0.5,2.5)--(-2.5,-0.5) (2.5,-1.5)--(1.5,-2.5);
     \end{tikzpicture}
     \caption{The line in red represents $y=0$, while the other three black lines
     are given by $z=\lambda x$,  $z=y+\lambda x$ and $z=\lambda x +y(1-\lambda)$.}
\end{figure}
Let $p$ denote the points of the configuration that lies on the line $z=\lambda x$. As we know, 
this point must lie on exactly three other lines of the configuration. However, we have four points 
available $(1:0:0), \ (0:1:1), \ (0:0:1), \ (1:1:1)$ that, a priori, would determine four additional lines. 
Therefore, it follows that $p$ must necessary be collinear with the line through two of these four points. 
Excluding the lines that already intersect the line $z=\lambda x$ at another point and those that already contain 
three points of the configuration, the only possibility is that $p$ lies on the line through $(1:1:1)$ and $(0:0:1)$, 
that is $x=y$. We obtain
\begin{equation*}
    p\in \{z=\lambda x\}\cap \{x=y\}\ \ \implies\ \ \ p=(1:1:\lambda)
\end{equation*}
Let $q$ denote the point on $z=y+\lambda x$. Using a similar reasoning, it follows that $q$ must also lie on 
the line through $(1:1:1)$ and $(0:1:0)$, that is 
\begin{equation*}
    q\in \{z=y+\lambda x\} \cap \{x=z\}\ \ \implies\ \ \ q=(1:1-\lambda:1)
\end{equation*}
Finally, let $r$ denote the third point on $z=\lambda x+y(1-\lambda)$. We obtain
\begin{equation*}
    r\in \{z=\lambda x+y(1-\lambda)\}
\cap \{z=0\}\ \ \ \ \implies \ \ \ \ r=(\lambda-1:\lambda:0)
\end{equation*}
The line through $p$ and $q$ must contain another point. The lines of the configuration for $(1:0:0)$ must be $4$. 
From the previous calculation, we have already considered $3$ of them, namely the lines for the points:
\[\begin{aligned}
    &(1:0:0) \ \ \ \ &(0:1:1) \ \ \ \ &(1:1:1)\\
    &(1:0:0) \ \ \ \ &(0:0:1) \ \ \ \ &(1:0:\lambda)\\
    &(1:0:0) \ \ \ \ &(0:1:0) \ \ \ \ &(\lambda-1:\lambda:0) 
\end{aligned}\]
It follows that the line through $p$ and the one through $q$ must coincide, that is, the three 
points $p$, $(1:0:0)$ and $q$ must be collinear. 
\begin{equation*}
    0=\left |\begin{matrix}
        1& 0&0\\
        1&1&\lambda\\
        1&1-\lambda&1
    \end{matrix}\right |=1+\lambda(\lambda-1)=\lambda^2-\lambda+1
\end{equation*}
The solutions are $\lambda_1=-\epsilon$ and $\lambda_2=\epsilon+1$ where $\epsilon$ is a primitive 
third root of unity. These two values give us the two possible configurations. The only thing left 
to verify is that $r$ also lies on the line through $q$ and $(0:0:1)$, and on the line through $p$ 
and $(0:1:1)$. \\
For $\lambda=-\epsilon$
\begin{equation*}
   \left | \begin{matrix}
        1&1+\epsilon&1\\
        0&0&1\\
        -\epsilon-1& -\epsilon &0
    \end{matrix}\right |= \epsilon-(\epsilon+1)^2=-(\epsilon^2+\epsilon+1)=0\ \ \ \ \ \ \ \ \ \ \ \ \ \ \ \ 
\end{equation*}
\begin{equation*}
    \left | \begin{matrix}
    1&1&-\epsilon\\
    0 &1&1\\
    -\epsilon-1&-\epsilon&0
    \end{matrix}\right |= -\epsilon(\epsilon+1)-(-\epsilon+(\epsilon+1))=-\epsilon^2-\epsilon-1=0
\end{equation*}
For $\lambda=1+\epsilon$
\begin{equation*}
   \left | \begin{matrix}
        1&-\epsilon&1\\
        0&0&1\\
        \epsilon& 1+\epsilon &0
    \end{matrix}\right |= -(1+\epsilon+\epsilon^2)=0\ \ \ \ \ \ \ \  \ \ \ \ \ \ \ \ \ \ \ \ \ \ \ \ \  \ \ \ \ \ \ \ \ \ \ \  \end{equation*}
\begin{equation*}
    \left | \begin{matrix}
    1&1&1+\epsilon\\
    0 &1&1\\
    \epsilon&1+\epsilon&0
    \end{matrix}\right |= -\epsilon(1+\epsilon)-(1+\epsilon-\epsilon)=-\epsilon^2-\epsilon-1=0
\end{equation*}\\

We have thus found the two configurations corresponding to Figure~\ref{figuradue}.\\
\begin{figure}[H]
\centering
\begin{minipage}{0.45\textwidth}
    \begin{tikzpicture}
        \fill[blue] (-1,0) circle (1pt) node[above left] {$(0:0:1)$};
       \fill[blue] (0,-1) circle (1pt) node[below] {$(1:1:1)$};
        \fill[blue] (0,1) circle (1pt) node[above] {$(1:0:0)$};
        \fill[blue] (1,0) circle (1pt) node[above right] {$(0:1:0)$};
        \fill[black] (0,0) circle (1pt) node[above] {$(0:1:1)$};  
        \fill[black] (-1,-1) circle (1pt) node[below left]{$(1:\epsilon+1:0)$};
        \fill[black] (1,-1) circle (1pt) node[below right] {$(1:0:-\epsilon)$};  
        \fill[black] (1,1) circle (1pt) node[above right]{$(1:1:-\epsilon)$}  ;
        \fill[black] (-1,1) circle (1pt) node [above left]{$(1:\epsilon+1:1)$};
        \end{tikzpicture}

     \end{minipage} \hspace{1cm}
     \begin{minipage}{0.45\textwidth}
     \centering
     \begin{tikzpicture}
        \fill[blue] (-1,0) circle (1pt) node[above left] {$(0:0:1)$};
       \fill[blue] (0,-1) circle (1pt) node[below] {$(1:1:1)$};
        \fill[blue] (0,1) circle (1pt) node[above] {$(1:0:0)$};
        \fill[blue] (1,0) circle (1pt) node[above right] {$(0:1:0)$};
        \fill[black] (0,0) circle (1pt) node[above] {$(0:1:1)$};  
        \fill[black] (-1,-1) circle (1pt) node[below left]{$(1:-\epsilon:0)$};
        \fill[black] (1,-1) circle (1pt) node[below right] {$(1:0:\epsilon+1)$};  
        \fill[black] (1,1) circle (1pt) node[above right]{$(1:1:1+\epsilon)$}  ;
        \fill[black] (-1,1) circle (1pt) node [above left]{$(1:-\epsilon:1)$};     
    \end{tikzpicture}      

     \end{minipage}
\end{figure}
\end{itemize}
\end{proof}
We denote by $Z$ the variety consisting of cubics whose Hessian \revised{map} is not defined, 
namely the cones (see \cite{Coni}). 
This variety has dimension $5$ and is contained in the variety of triangles, denoted by $Tr$. 
Here, \emph{triangles} are cubics that decompose as the union of three lines. 
The variety $Tr$ can be defined as the closure of the set of cubics that coincide with their 
Hessian and has dimension $6$.
 
\begin{prop}\label{triangoliper6punti}
    Given six general points in $\mathbb{P}^2$, meaning that no three of them are collinear, 
    there exist exactly $15$ triangles \revised{each passing through all these points.} 
\end{prop}
\begin{proof}
    \begin{figure}[H]
\centering
    \begin{tikzpicture}

       \fill[black] (-2,-1) circle (2pt) node[above left] {$p_2$};
        \fill[black] (0,-1.5) circle (2pt) node[above left] {$p_4$};
        \fill[black] (2,-1) circle (2pt) node[above left] {$p_6$};
       \fill[black] (-2,0) circle (2pt) node[above left] {$p_1$};
        \fill[black] (0,0.5) circle (2pt) node[above left] {$p_3$};
        \fill[black] (2,0) circle (2pt) node[above left] {$p_5$};

     \end{tikzpicture}
\end{figure}
\revised{Triangles of this type are in bijection with partitions of the set
\(\{p_1,\dots,p_6\}\) into three unordered pairs, since each pair of points
determines one of the three lines defining the triangle. Hence, the number of 
such triangles equals the number of partitions of a
6-element set into three unordered pairs, namely
$\frac{6!}{(2!)^3\,3!}=15$.}
\end{proof}

The special linear group $SL(3)$ acts on the space of plane cubics, and the Hessian map is equivariant 
with respect to this action \revised{(see \cite[Prop.~4.4.2]{Invariantstrumfels})}. Smooth cubic curves lie in infinitely many orbits. Besides these, there are 
eight additional orbits containing singular cubics (Table $2.2$ in \cite{OrbitCubicsSL3}). \\
From Theorem $5.7$ in \cite{Hessianmap} we know that the Hessian map for plane cubics is generically [3:1], 
meaning that a general smooth cubic is the Hessian of three other cubics. 
In the following table, 
we describe the locus $H^{-1}(f)$ for all types of cubics, choosing one representative in each orbit;
\revised{this follows from a direct computation}. 
\begin{table}[H]
\begin{adjustwidth}{-2cm}{-2cm}
\centering
\renewcommand{\arraystretch}{1.5}
    \begin{tabular}{c|c|c}
    \text{\textbf{Cubic $f$}} & \textbf{dim$(\overline{H^{-1}(f)})$}& \textbf{Description of $g\in 
    H^{-1}(f)$ }\\\hline
      $x^3$   &4& $a_0x^3+3a_1x^2y+3a_2x^2z+3a_3xy^2+6a_4xyz+3a_5xz^2$ \\ 
      & & $with \ \ \ a_4^2=a_3a_5$\ \ \ $and \ \ \ a_2^2a_3-2a_1a_2a_4+a_1^2a_5\neq 0$\\ \hline 
       $xy(x+y)$ &$empty$ & \\ \hline
       $x^2y$& 3&$a_0x^3+3a_1x^2y+3a_2x^2z+a_6y^3\ \ with \ \ a_2,a_6\neq 0$\\ \hline
       $x(x^2+yz)$&0& $x^3-3xyz$\\ \hline
       $y(x^2+yz)$ & $empty$&\\ \hline
       $y^2z-x^3-x^2z$ &0& $-2x^3-3x^2z+3xy^2+3y^2z$\\ \hline
       $y^2z-x^3$ & $empty$&\\ \hline
       $x^3+y^3+z^3-3txyz$ \text{ with } $t^3\neq 1$ & 0 & $x^3+y^3+z^3-3\lambda xyz\ \ with\ \ 4-\lambda^3=3\lambda^2t$ \\ \hline
       $xyz$ &2& $a_0x^3+a_6y^3+a_9z^3$\ \ with  $a_0,a_6,a_9\neq 0$ \  \ \textit{and}\ \ \ $xyz$\\ 
    \hline 
    \end{tabular}
    \caption{Description of the preimage of the Hessian map.}
     \label{tab:hessianspace}
\end{adjustwidth}

\smallskip
\noindent
Geometrically, the cubics in $H^{-1}(x^3)$ are unions of a nonsingular conic with its 
tangent line $x=0$; those in $H^{-1}(x^2y)$ are cuspidal cubics with cusp at $[0:0:1]$; 
and those in $H^{-1}(xyz)$ consist of smooth cubics of the form $a_0x^3+a_6y^3+a_9z^3$ 
together with the reducible cubic $xyz$.
\end{table}

\subsection{The Hesse Pencil Variety of Plane Cubics}
\begin{defn}[The Hesse Pencil]
    Let $f$ be a smooth cubic. The pencil generated by $f$ and its Hessian $H(f)$, that is 
    \begin{equation*}
        \langle f,H(f)\rangle =\lambda f + \mu H(f)  \ \ \ \ \ \ \ \ \lambda, \mu \in \mathbb{C},
    \end{equation*} is called  $ Hesse\ pencil$. 
\end{defn}

Observe that every cubic in such a pencil passes through the nine inflection points of $f$. 
Since a pencil of cubics is determined by eight points, it follows that each Hesse pencil 
is characterized by passing through nine points in Hesse configuration, which are therefore 
inflection points for all cubics in the pencil. These nine points will be referred to as 
$the\ base \ points$ of the pencil. \\
Hesse’s Theorem (see \cite[Vol.~3, Ch.~3]{Enriquescubiche}) states that such pencils of 
cubics contain the Hessian of every cubic in the pencil; moreover, each smooth cubic in 
the pencil arises as the Hessian of three other cubics belonging to the same Hesse pencil.
\begin{prop}\label{Hessepenciltriangoli}
    Every cubic in the variety of triangles $Tr$ belongs to a $2$-dimensional family of Hesse pencils. 
\end{prop}
\begin{proof}
    Observe that for any cubic $T\in Tr$, the preimage under the Hessian map consists of all 
    cubics whose Hessian is $T$. Each such cubic, together with $T$, defines a Hesse pencil. 
    By the dimension count reported in Table~\ref{tab:hessianspace}, this preimage is $2$-dimensional. 
\end{proof}

\defXX*
From the characterization of orbits under the action of $SL(3)$, it follows that every smooth 
cubic lies in the orbit of a cubic of the form $x^3+y^3+z^3+6txyz$ with $t^3\neq 1$. 
The Hessian of such a cubic, when $t\neq 0$, is of the same form with $t'=\frac{-1-2t^3}{6t^2}$. 
Hence, the Hesse pencil generated by these cubics coincides with the pencil generated by 
the Fermat cubic and its Hessian, namely $\langle x^3+y^3+z^3,xyz\rangle$. Moreover, every 
Hesse pencil belongs to the orbit of this one under the action of $SL(3)$, extended to 
pencils of cubics, and $H_8$ coincides with the closure of the orbit of the pencil 
$\langle x^3+y^3+z^3,xyz\rangle$, that is
\begin{equation*}
    H_8=\overline{O(\langle x^3+y^3+z^3,xyz\rangle)}
\end{equation*}
\begin{prop}
    The Hesse Pencil Variety $H_8$ has dimension 8. 
\end{prop}
\begin{proof}
    Consider 
    \[
        P := \overline{\{(f,L) \in \mathbb{P}^9 \times H_8 \mid f \in L\}} \subset \mathbb{P}^9 \times G(1,9),
    \]
    with projections $p_1: P \to \mathbb{P}^9$ and $p_2: P \to H_8$.  

    For a generic $f\in \mathbb{P}^9$, the fiber $p_1^{-1}(f)$ is finite, consisting of the 
    unique pencil $L = \langle f,H(f)\rangle$, so $\dim(P)=9$.  
    For $L\in H_8$, the fiber $p_2^{-1}(L)$ is the line $L$ itself, hence $\dim(p_2^{-1}(L))=1$.  
    It follows that
    \[
        \dim(P) = \dim(H_8) + 1 = 9 \quad\Rightarrow\quad \dim(H_8) = 8.
    \]
\end{proof}

In $G(1,9)$ there are exactly five Schubert cycles of dimension $8$. These cycles \revised{generate} 
$H^{16}(G(1,9),\mathbb{Z})$. The degree of $H_8$ can be expressed as a linear combination 
of the degree of these cycles. Each of these cycles corresponds to a tableau consisting 
of $8$ squares contained within a $2\times 10$ matrix. In particular, these are: 
\[
\begin{array}{cc}
\text{$\lambda_1$=\ \ }
    \begin{ytableau}
      \empty  &\empty  & \empty&\empty&\empty& \empty &\empty &\empty\\
    \end{ytableau}
    & \quad\ \ \ \ 
    \text{$\lambda_2=$\ \ }
    \begin{ytableau}
      \empty  &\empty  & \empty&\empty&\empty& \empty &\empty\\
      \empty \\
    \end{ytableau}
\end{array}\]

\[
\begin{array}{ccc}
\text{$\lambda_3$=\ \ }
    \begin{ytableau}
      \empty  &\empty  & \empty&\empty&\empty& \empty\\
      \empty &\empty \\
    \end{ytableau}

    & \quad\ \ \ \ \
    \text{$\lambda_4=$\ \ }
    \begin{ytableau}
      \empty  &\empty  & \empty&\empty&\empty\\
      \empty&\empty &\empty \\
    \end{ytableau}
    & \quad\ \ \ \
    \text{$\lambda_5=$\ \ }
    \begin{ytableau}
        \empty  &\empty  & \empty&\empty\\
      \empty&\empty &\empty& \empty \\
      \end{ytableau}
\end{array}\] 

If $\lambda_i$ represents one of these diagrams, we denote by $X_i$ the corresponding Schubert variety. 
Using the notation in \cite{Schubertcalculus}, these Schubert cycles correspond, in order, 
to $\Omega(0,9)$, $\Omega(1,8)$, $\Omega(2,7)$, $\Omega(3,6)$, $\Omega(4,5)$. We have
\begin{equation}\label{gradoSformula}
    \deg(H_8)=\sum_{i=1}^5 \beta_i\ \deg(X_i)
\end{equation}
The vector $\beta=(\beta_1,\beta_2,\beta_3,\beta_4,\beta_5)$ represents the multidegree of $H_8$. 
Moreover, we also know that the coefficients $\beta_i$ are found as the number of elements in the 
intersection between \revised{$H_8$} and the Schubert variety dual to $X_i$. However, since all these 
varieties are self-dual, the intersection can be taken with $X_i$ itself.
\begin{prop}\label{multigradoScubiche}
    Using the notation above, the multidegree of the variety $H_8$ is given by 
    \begin{equation*}
        \beta=(1,3,9,12,6)
    \end{equation*}
    In particular, the degree of $H_8$ is $622$.
\end{prop}
\begin{proof}
\begin{itemize}
\item[-] $\Omega(0,9)$\\
This Schubert cycle parametrizes all lines in $\mathbb{P}^9$ through a fixed point $f$. 
Intersecting with $H_8$ amounts to asking how many Hesse pencils contain $f$. 
By Hesse’s theorem in \cite{Enriquescubiche}, the answer is unique: the pencil $\langle f,H(f)\rangle$. 
Hence $\beta_1=1$.
\item[-] $\Omega(1,8)$\\
This cycle parametrizes lines in $\mathbb{P}^9$ that meet a fixed line $l$ and lie in a 
hyperplane $\mathbb{P}^8$ containing $l$. 
Intersecting with $H_8$ means counting Hesse pencils with this property, i.e.\ cubics on 
$l$ whose Hessian lies in the given hyperplane. 
Writing $a_i=p_i+\lambda v_i$, the coefficients of the Hessian, $\overline{a_i}$, 
are cubic polynomials in $\lambda$. 
The hyperplane condition 
\[
\sum_{i=0}^9 \mu_i \overline{a_i}=0
\]
is therefore a cubic equation in $\lambda$, admitting three solutions. 
Hence $\beta_2=3$.

\item[-] $\Omega(2,7)$\\ 
This cycle parametrizes lines in $\mathbb{P}^9$ meeting a fixed plane $\pi$ and 
contained in a $\mathbb{P}^7$ that also contains $\pi$. 
Thus, we need to count cubics in $\pi$ whose Hessian lies in the given $\mathbb{P}^7$. 
Writing 
\[
a_i=p_i+\lambda v_i+\mu w_i, \qquad i=0,\dots,9,
\]
the Hessian coefficients $\overline{a_i}$ are cubic polynomials in $(\lambda,\mu)$. 
If the $\mathbb{P}^7$ is cut out by two hyperplanes, the condition becomes
\[
\sum_{i=0}^9 \delta_i \overline{a_i}=0, \qquad 
\sum_{i=0}^9 \gamma_i \overline{a_i}=0,
\]
namely two cubic equations in $(\lambda,\mu)$. 
By Bézout's theorem there are $9$ solutions, hence $\beta_3=9$.

\item[-] $\Omega(3,6)$\\
Considering cubics in a $\mathbb{P}^3$ whose Hessian 
lies in a $\mathbb{P}^6$, Bézout's theorem gives $27$ solutions. However, this count 
also includes the cubics in the variety of triangles $Tr$, since $\dim(Tr)=6$ and thus 
a generic $\mathbb{P}^3$ intersects it nontrivially. By Proposition \ref{triangoliper6punti}, 
the number of triangles passing through six generic points-hence contained in a 
generic $\mathbb{P}^3$-is $15$. We subtract these to count only the Hesse pencils, 
giving $27-15=12$. 
\revised{It is important to observe that, although a triangle does not generate a Hesse pencil, 
it is contained in a $2$-dimensional family of Hesse pencils, (see Proposition~\ref{Hessepenciltriangoli}). 
A priori, one might expect additional pencils inside the given $\mathbb{P}^6$ 
arising from these families. However, this does not occur: a $2$-dimensional 
family in $\mathbb{P}^9$ has empty intersection with a generic $\mathbb{P}^6$. 
Hence no further pencils appear beyond the $12$ already counted.}
This justifies that we do not count any further 
pencils beyond the $12$ already identified. Hence $\beta_4=12$

\item[-] $\Omega(4,5)$\\
This cycle parametrizes lines in $\mathbb{P}^9$ intersecting a fixed $\mathbb{P}^4$ and 
contained in a $\mathbb{P}^5$ through it. 
Since a general line in $\mathbb{P}^5$ always meets $\mathbb{P}^4$, this reduces to counting 
the Hesse pencils contained in a general $\mathbb{P}^5\subset \mathbb{P}^9$. 
We may regard such a $\mathbb{P}^5$ as the space of cubics passing through four fixed points. 
Therefore, we are asking for the number of cubics whose Hessian also passes through these four points. 
Equivalently, these are cubics with four prescribed inflection points, i.e.\ Hesse 
configurations of nine points containing four fixed ones. 
By Proposition~\ref{configurazioni4punti}, there are six such configurations, hence $\beta_5=6$.
\end{itemize}

Finally, to compute the degree of \revised{$H_8$} it suffices to use the formula (\ref{gradoSformula}), and we obtain
\begin{equation*}
    \deg(H_8)=1\cdot1+3\cdot 7+9\cdot 20+12\cdot 28+6\cdot14=622
    \end{equation*}
    \end{proof}

\subsection{A Skew Invariant of Plane Cubics}\label{sectionR}
We identify the space of plane cubic forms with $\mathrm{Sym}^3(\mathbb{C}^3)$ and consider 
the exterior power $\bigwedge^3(\mathrm{Sym}^3(\mathbb{C}^3))$. Using the \texttt{SchurRings} 
package in \texttt{Macaulay2} (\cite{M2}), this decomposes into Schur modules:
\[
    \bigwedge\nolimits^3(\mathrm{Sym}^3\mathbb{C}^3)
    = \mathbb{S}_{(7,1,1)} \oplus \mathbb{S}_{(6,3)} \oplus \mathbb{S}_{(5,3,1)} \oplus \mathbb{S}_{(3,3,3)},
\]
with dimensions $(15,21,15,1)$, respectively. The summand $\mathbb{S}_{(3,3,3)}$ corresponds 
to a skew-invariant of degree $3$, whose explicit expression can be obtained via the symbolic 
method (see \cite{Invariants}, \cite{Invariantstrumfels}).
\\
In the language of tableau functions, this skew-invariant, which we denote by $R$, corresponds
to the $SL(3)$-invariant function associated with the $3 \times 3$ tableau
\[
\ytableausetup{mathmode, boxsize=1.5em}
\begin{ytableau}
    1 & 1 & 1 \\
    2 & 2 & 2 \\
    3 & 3 & 3
\end{ytableau}.
\]
We define $R$ on the Veronese variety by
\[
    R \colon \mathrm{Sym}^3(\mathbb{C}^3) \times \mathrm{Sym}^3(\mathbb{C}^3) \times \mathrm{Sym}^3(\mathbb{C}^3) \longrightarrow \mathbb{C},
\]
\[
    R(l^3,m^3,n^3) := (\,l \wedge m \wedge n\,)^3 ,
\]
where $l,m,n$ are linear forms, and $l \wedge m \wedge n$ denotes their determinant. 
From this definition it is immediate that $R$ is an alternating invariant.

The explicit expression of $R$ on three generic cubic forms can be computed using 
the following \texttt{Macaulay2} script.

\begin{lstlisting}[caption={Macaulay2 code computing the invariant $R$}, label={lst:M2inv}]
KK=QQ
R1=KK[x_1..x_3,y_1..y_3,z_1..z_3,a_0..a_9,b_0..b_9,c_0..c_9]
inv=det(matrix{{x_1,x_2,x_3},{y_1,y_2,y_3},{z_1,z_2,z_3}})^3
syma=(x,h)->(contract(x,h)*transpose matrix{{a_0..a_9}})_(0,0)
symb=(x,h)->(contract(x,h)*transpose matrix{{b_0..b_9}})_(0,0)
symc=(x,h)->(contract(x,h)*transpose matrix{{c_0..c_9}})_(0,0)
invx=syma(symmetricPower(3,matrix{{x_1,x_2,x_3}}),inv)
invy=symb(symmetricPower(3,matrix{{y_1,y_2,y_3}}),invx)
invz=symc(symmetricPower(3,matrix{{z_1,z_2,z_3}}),invy) 
\end{lstlisting}

Here $x_i,y_i,z_i$ are the coefficients of the three linear forms, while $(a_i), (b_i), (c_i)$ 
represent the coefficients of three generic cubics.

The computation yields the following explicit form for $R$, which consists of a sum of 54 
monomials and can be expressed in terms of wedge products:\\

 $R((a_i),(b_i),(c_i))=(-1)(a_9\wedge a_6 \wedge a_0-3a_8\wedge a_7\wedge a_0-3a_9\wedge a_3 \wedge a_1+
      6a_8\wedge a_4 \wedge a_1+3a_7\wedge a_5\wedge a_1 +3a_8\wedge a_3\wedge a_2-6a_7\wedge a_4\wedge a_2-3a_6\wedge a_5\wedge a_2-6a_5\wedge a_4\wedge a_3)$\\

As we know, the Hessian of a plane cubic is also a cubic, hence it can be represented by a 
vector of its ten coefficients $\overline{a}=(\overline{a}_0,\ldots,\overline{a}_9)$, which 
are homogeneous polynomials of degree $3$ in the variables $a_i$. Thus $\overline{a}\in 
(\mathbb{C}[a_0,\ldots,a_9]_3)^{10}$, and a syzygy among the $\overline{a}_i$ is a $10$-tuple 
$s=(s_0,\ldots,s_9)$ in $\mathbb{C}[a_0,\ldots,a_9]$ such that $\sum_i \overline{a}_i s_i=0$.
\texttt{Macaulay2} (\cite{M2}) computes the module of syzygies, which is generated by $45$ 
elements: $10$ of them linear in the $a_i$ and the remaining $35$ cubic. 
We focus on the linear ones, as they play a crucial role in what follows.  \\

We now associate to the invariant $R$ a matrix $\overline{R}$, defined so that  
\[
R((b_i),(a_i),(c_i)) \;=\; b \,\overline{R}\, c ,
\]  
where $a=(a_0,\ldots,a_9)$, $b=(b_0,\ldots,b_9)$ and $c=(c_0,\ldots,c_9)$ represent the 
coefficients of three generic cubic forms. The entries of $\overline{R}$ are linear 
in the $a_i$, and it can be written explicitly as
\[
\overline{R}= 
\begin{bmatrix}
            0  &   0   &  0  &   0  &   0  &  0  &   a_9& -3a_8 & 3a_7 &-a_6  \\
       0 &    0  &   0 &    -3a_9&  6a_8 &-3a_7 &  0 &    3a_5 & -6a_4 &  3a_3 \\
       0  &   0  &   0 &    3a_8&  -6a_7&  3a_6& -3a_5&  6a_4&  -3a_3&   0     \\
      0   &  3a_9& -3a_8 &  0  &   -6a_5 & 6a_4&  0  &   0  &   3a_2 & -3a_1  \\
       0  &   -6a_8 &  6a_7 & 6a_5&  0  &  -6a_3 &  0  &   -6a_2 &  6a_1 & 0     \\
      0   &  3a_7  &-3a_6 & -6a_4 &  6a_3& 0  &   3a_2 &-3a_1 &  0  &   0     \\
       -a_9 & 0   &  a_5 &0 &    0 &   -3a_2 & 0 &    0  &   0   &  a_0 \\
      3a_8 &-3a_5  & -6a_4  & 0   &  6a_2 &3a_1 & 0  &   0  &   -3a_0 & 0     \\
      -3a_7 & 6a_4 & 3a_3&  -3a_2 &  -6a_1 & 0   &  0    & 3a_0 &0   &  0     \\
      a_6& -3a_3 & 0   &  3a_1& 0   & 0    & -a_0 & 0  &   0  &   0     \\
\end{bmatrix}
\]

\begin{prop}
Let $R$ be the antisymmetric invariant of $Sym^3(\mathbb{C}^3)$ discussed earlier. 
Then, for any $f\in Sym^3(\mathbb{C}^3)$ whose Hessian is defined, we have
\[
R(f,H(f),-)=0 ,
\]
where $H(f)$ denotes the Hessian of $f$.
\end{prop}  

\begin{proof}
The columns of $\overline{R}$ coincide (up to permutation) with the linear syzygies 
among the coefficients of the Hessian. Hence, substituting $f=(a_i)_i$ and $H(f)=(\overline{a}_i)_i$, we obtain  
\[
-R(f,H(f),-)=\overline{a}\,\overline{R}(-)=0,
\]  
as claimed.
\end{proof}
The following \texttt{Macaulay2} script continues the one in Listing~\ref{lst:M2inv} and can be used to compute
the matrix $\overline{R}$ and the syzygies of the Hessian.

\begin{lstlisting}[
  caption={Macaulay2 code computing $\overline{R}$ and the syzygies of the Hessian},
  label={lst:M2RbarHessSyzygies},
  basicstyle=\ttfamily\small
]
X = matrix{{x_1^3, (1/3)*x_1^2*x_2, (1/3)*x_1^2*x_3, (1/3)*x_1*x_2^2, 
    (1/6)*x_1*x_2*x_3,(1/3)*x_1*x_3^2, x_2^3, (1/3)*x_2^2*x_3, 
    (1/3)*x_2*x_3^2, x_3^3}}
x1 = matrix{{x_1..x_3}}

-- matrix \overline{R}
minv = diff(matrix{{c_0..c_9}}, diff(transpose matrix{{b_0..b_9}}, invz))
rank minv -- 8

-- syzygies among the coefficients of the Hessian
f = x_1^3*a_0 + 3*x_1^2*x_2*a_1 + 3*x_1^2*x_3*a_2 + 3*x_1*x_2^2*a_3 
    + 6*x_1*x_2*x_3*a_4+ 3*x_1*x_3^2*a_5 + x_2^3*a_6 + 3*x_2^2*x_3*a_7 
    + 3*x_2*x_3^2*a_8 + x_3^3*a_9
hf  = det diff(x1, diff(transpose x1, f))
xhf = contract(X, hf) -- Hessian coefficients
kernel xhf
ghf = gens kernel(xhf) -- 10x45 matrix, columns are syzygies of \overline{a}
lhf = submatrix(ghf,,{0..9}) -- the linear 10x10 submatrix
\end{lstlisting}
\corfg*

\begin{proof}
    This result follows directly from the previous theorem and the linearity of the invariant $R$. 
\end{proof}

We can rewrite $R$ as a scalar product between two vectors of length $10$:
\[
R = n(a_i,b_i) \cdot (c_0,\ldots,c_9),
\]
where $(c_0,\ldots,c_9)$ represents a cubic form in $Sym^3(\mathbb{C}^3)$, and 
$n(a_i,b_i) \in (\mathbb{C}[a_i,b_i])^{10}$.  

The vector $n(a_i,b_i)$ is closely related to the Pl\"ucker coordinates of the line 
through the points $a_i$ and $b_i$. Specifically, letting
\[
p_{ij} = a_i b_j - a_j b_i, \quad \forall i=0,\ldots,9, \ j>i,
\]
the components of $n$ can be rewritten in terms of these coordinates:\\

$n(p_{ij})=(3p_{7,8}-p_{6,9},\,3p_{5,7}-6p_{4,8}+3p_{3,9},\,3p_{5,6}-6p_{4,7}+3p_{3,8},\,6p_{4,5}+3p_{2,8}-3p_{1,9},\,
6p_{3,5}+6p_{2,7}-6p_{1,8},\,3p_{2,5}-p_{0,9},\,6p_{3,4}+3p_{2,6}-3p_{1,7},\,6p_{2,4}+3p_{1,5}-3p_{0,8},\,
3p_{2,3}+6p_{1,4}-3p_{0,7},\,3p_{1,3}-p_{0,6})$.\\

\begin{defn}\label{defN}
We define $N \subset \mathbb{P}^{44}$ to be the subvariety of $G(1,9)$
whose Plücker coordinates $p_{ij}$ satisfy the vanishing of the entries
of the vector $n(p_{ij})$, together with the Plücker equations defining
the Grassmannian $G(1,9)$.
\end{defn}
Using \texttt{Macaulay2}~\cite{M2}, 
we compute the dimension and degree of $N$, as well as the degrees of 
the generators of its defining ideal.
\begin{equation}\label{gradoN}
    \dim(N)=8, \quad \deg(N)=622,
\end{equation}
with ideal generated by $10$ linear and $200$ quadratic polynomials:
\begin{verbatim}
     betti mingens N

            0   1
     total: 1 210
         0: 1  10
         1: . 200
\end{verbatim}

Let us now consider the space $\bigwedge^2Sym^3\mathbb{C}^3$, which has dimension 45. 
We can think of $G(1,9)\subset \mathbb{P}(\bigwedge^2(Sym^3\mathbb{C}^3))$. 
This space decomposes into Weyl modules as follows:
\begin{equation}\label{decompcubiche}
    \bigwedge\nolimits^2(Sym^3\mathbb{C}^3)=\mathbb{S}_{(5,1)}\mathbb{C}^3\ \oplus \ \mathbb{S}_{(3,3)}\mathbb{C}^3
\end{equation}
These two covariants have dimension $35$ and $10$, respectively. 
\revised{The following statement is the natural counterpart of Proposition~\ref{propA-F} for binary quartics.}

\begin{prop}
    Using the previous notations, we have 
    \begin{equation*}
       \mathbb{P}\left( \widehat{G(1,9)}\cap \mathbb{S}_{(5,1)}\mathbb{C}^3\right)=N.
    \end{equation*}
\end{prop}

\begin{proof}
The variety $N$ is invariant under the action of $SL(3)$. 
The ten linear equations defining $N$ span a $10$–dimensional $SL(3)$–stable subspace. 
By the decomposition in \ref{decompcubiche}, $\mathbb{S}_{(3,3)}\mathbb{C}^3$ 
is the unique $10$–dimensional $SL(3)$–invariant subspace. 

Hence, the vanishing of these ten linear forms defines the dual subspace, i.e. 
it projects onto the complementary summand $\mathbb{S}_{(5,1)}\mathbb{C}^3$. 
Intersecting the affine cone $\widehat{G(1,9)}$ with this subspace and 
projectivizing gives precisely $N$, as claimed.
\end{proof}

The variety $N$ contains the lines of $\mathbb{P}^9$ for which $R(p,q,-)=0$, where $p$ and $q$ are 
two points of the line itself. Remember that these points in $\mathbb{P}^9$ correspond to cubics, 
and according to corollary \ref{R(f,H(f),-)=0}, a generic point of this variety corresponds to a 
Hesse pencil. It follows that 
\begin{equation*}
    H_8\subset N,
\end{equation*}
\revised{where $H_8$ has been defined in Definition~\ref{defXX}.}
Analogously to what was done for the variety $H_8$, the degree of $N$ can also be 
expressed as a weighted sum of the degrees of the $8$-dimensional Schubert varieties in $G(1,9)$, 
\begin{equation*}
    \deg(N)=\sum_{i=1}^{5}\alpha_i\ \deg(X_i).
\end{equation*}
The vector $\alpha=(\alpha_1,\alpha_2,\alpha_3,\alpha_4,\alpha_5)$ represents the 
multidegree of $N$, which we have computed using \texttt{Macaulay2} (\cite{M2}). 
In order to determine each coefficient $\alpha_i$, one considers the intersection of 
$N$ with a Schubert variety of type $X_i$. Care must be taken to choose the Schubert 
variety generically so that the intersection is zero-dimensional. \revised{This can be achieved
by following the approach used in \cite[Proposition~4]{Schubertcalculus}, taking as $T$ 
a random $10\times 10$ matrix and apply this linear transformation to a 
classically defined Schubert variety.}
\begin{oss}
Using the notation of \cite{Schubertcalculus}, a Schubert variety corresponding to 
$\Omega(i,j)$ is defined with respect to a flag $A_0 \subset A_1 \subset \mathbb{P}^9$,
where $A_0$ and $A_1$ are linear subspaces of dimensions $i$ and $j$, respectively.

By a \emph{classically defined Schubert variety} we mean the one obtained by taking 
$A_0$ and $A_1$ to be the coordinate subspaces defined by the vanishing of the last 
$i$ and $j$ homogeneous coordinates.

To construct this variety in \texttt{Macaulay2}, introduce auxiliary variables 
$a_0,\dots,a_9$ representing a general point of $\mathbb{P}^9$, and consider a 
$2\times 10$ matrix $m$ whose rows correspond to general points of $A_0$ and $A_1$, 
written in terms of the $a_i$. The Plücker coordinates are given by the $2\times 2$ 
minors of $m$. This defines an ideal in the polynomial ring generated by the 
Plücker variables together with the auxiliary variables $a_0,\dots,a_9$. 
Eliminating the auxiliary variables yields the ideal defining the classical 
Schubert variety $\Omega(i,j)$.

To obtain a Schubert variety of the same type in general position, let $T$ be a 
$10\times 10$ matrix with random entries and replace $m$ with $mT$. The Plücker 
coordinates are computed as the $2\times 2$ minors of $mT$. Eliminating the 
auxiliary variables again produces a Schubert variety of type $\Omega(i,j)$ in 
general position.
\end{oss}
\renewcommand{\arraystretch}{1.5}
\begin{table}[H]

\centering
    \begin{tabular}{c|c|c|c}
    Schubert variety $X_{\lambda}$ & $ \deg(X_\lambda)$& $codim(N\cap X_{\lambda})$& $\deg(N\cap X_{\lambda})$\\
    \hline
      (0,9)   &$1$& 44& 1\\ \hline
       (1,8)  & $7$ &44&3\\ \hline
       (2,7) &$20$& 44&9\\ \hline
       (3,6)& $28$&44 &12\\ \hline
       (4,5)& $14$& 44& 6\\

    \hline 
    \end{tabular}
    \caption{Multidegree of $N$}
    \label{tab:multigradoN}
\end{table}

We thus obtain that $\alpha=(1,3,9,12,6)$, and using the formula above for the degree of 
$N$, we get $\deg(N)=1\cdot1+3\cdot 7+9\cdot 20+12\cdot 28+6\cdot14=622$, which indeed 
matches the degree of $N$ previously found directly from the equations, see (\ref{gradoN}). \\

Since $R$ is an invariant under the action of $SL(3)$, and the variety $N$ is defined in 
term of $R$, it follows that $N$ must be invariant under this action, meaning that it is 
composed by $SL(3)$-orbits of pencils. \revised{The last part of this section is devoted to a 
classification of the orbits contained in $N$. 
This result will be fundamental for the proof of the main theorem; in particular, 
combined with the analysis of the 
multidegree of Schubert varieties, it will enable us to prove the equality $N = H_8$.}\\

We begin by defining, for each cubic $f\in \mathbb{P}^9$,
\begin{equation}
    \mathbb{P}_f^8:=\overline{\{\langle f,g\rangle \in G(1,9)\ |\ g\in \mathbb{P}^9}\}
\end{equation}
which is the closure of the set of pencils through $f$. This space has dimension $8$. \\
For each orbit of cubics, we fix a representative and compute the dimension of 
$\mathbb{P}_f^8\cap N $. The results are summarized in the following table, along with 
a description of the corresponding space.

\renewcommand{\arraystretch}{1.72}
\begin{table}[H]
\begin{adjustwidth}{-3cm}{-3cm}

\centering
    \begin{tabular}{c|c|c}
    \textbf{Cubic $f$} & \textbf{dim($\mathbb{P}_f^8\cap N$)}&\textbf{Description}\\
    \hline
      $x^3$   &$4$& $\langle x^3,3b_1 x^2y+3b_2 x^2z+3b_3xy^2+6b_4 xyz+3b_5 xz^2\rangle $\\ \hline
       $xy(x+y)$  & $2$&$\langle xy(x+y),b_0x^3+3b_2(x^2z+xyz+y^2z)+b_6y^3\rangle $ \\ \hline
       $x^2y$ &$2$&$\langle x^2y,b_0x^3+3b_2x^2z+b_6y^3\rangle $\\ \hline
       $x(x^2+yz)$&$0$& $\langle x^3,xyz\rangle $\\ \hline
       $x(y^2+xz)$&$0$&$\langle x(y^2+x z),x^3\rangle$\\ \hline
       $y^2z-x^3-x^2z$&$0$& $\langle y^2z-x^3-x^2z,3xy^2-x^2z+y^2z\rangle $\\ \hline
       $y^2z-x^3$ &$0$& $\langle y^2z-x^3,xy^2\rangle$\\ \hline
       $x^3+y^3+z^3-3txyz$ \text{ with } $t^3\neq1$& $0$&$\langle x^3+y^3+z^3,xyz\rangle$\\ \hline
       $xyz$ &$2$&$\langle xyz,b_1x^3+b_6y^3+b_9z^3\rangle$ \\ 
    \hline 
    \end{tabular}
    \caption{Orbit representatives of cubics and the corresponding pencils contained in $N$.}
    \label{tab:pencilN}
    \end{adjustwidth}
\end{table}
Note that the only cubic $f$ for which a single pencil arise are those for which the 
Hessian is defined and it is distinct from the cubic $f$; in this case, the pencil is 
generated by the cubic itself and its Hessian. 
\revised{Furthermore, to study the singular locus 
of \(N\), we consider the \(210 \times 45\) Jacobian matrix obtained from the equations defining \(N\), 
and we evaluate its rank. Since \(N\) has dimension \(8\), if the rank is lower than \(36\), 
then the point is singular. We also report the dimension of each orbit.
Both the rank computations and the orbit dimensions were obtained
computationally using \texttt{Macaulay2}. The dimension of each orbit was computed as 
the rank of the Jacobian matrix 
of the orbit map $SL(3)\to \mathcal{O}(f)$ evaluated at the identity.}
\renewcommand{\arraystretch}{1.72}
\begin{table}[H]
\begin{adjustwidth}{-3cm}{-3cm}

\centering
    \begin{tabular}{c|c|c}
    \textbf{Pencil} $\langle f,g\rangle$ & \textbf{dim($O(\langle f,g\rangle )$)}&\textbf{Rank $J$}\\
    \hline
      $\langle x^3+y^3+z^3,xyz\rangle$   &$8$& $36$\\ \hline
       $\langle y^2z-x^3-x^2z,3xy^2-x^2z+y^2z\rangle$  & $7$&$36$ \\ \hline
       $\langle y^2z-x^3,xy^2\rangle$ &$6$&$36$\\ \hline
       $\langle x^3,xyz\rangle$&$6$& $36$\\ \hline
       $\langle x(y^2+xz),x^3\rangle$&$5$&$36$\\ \hline 
    \end{tabular}
    \caption{Pencils of cubics contained in $N$ that are of the form $\langle f,H(f)\rangle $}
    \label{tab:HF-list}
    \end{adjustwidth}
\end{table}
\begin{oss}\label{pencil<f,H(f)>}
    \revised{Whenever a pencil in \(N\) contains a cubic \(f\) whose Hessian \(H(f)\)
is defined and different from \(f\), Table~5 shows that for all the orbits
containing such cubics (namely,
\(x(x^2+yz)\),
\(x(y^2+xz)\),
\(y^2z-x^3-x^2z\),
\(y^2z-x^3\),
\(x^3+y^3+z^3-3txyz\) with \(t^3\neq 1\))
one has $\dim\bigl(\mathbb{P}^8_f \cap N\bigr)=0$.
Moreover, Table~5 also shows that the unique pencil contained in \(N\) and passing
through such a cubic \(f\) is the pencil \(\langle f, H(f)\rangle\).
In particular, the only pencil in \(N\) containing \(f\) belongs to one of the
orbits listed above (Table~6); this last fact also follows from the equivariance of
the Hessian map with respect to the action of \(SL(3)\).}
\end{oss}
\begin{prop}\label{numerofinitoorbite}
For every cubic, the pencils through it that are contained in $N$ form only finitely many $SL(3)$–orbits. 
In particular, $N$ contains a finite number of orbits of pencils \revised{and
its singular locus is precisely the union of the two orbits 
\[
O(\langle x^3,x^2y\rangle) \quad \text{and} \quad O(\langle x^2y,x^2z\rangle).
\]}
\end{prop}
Before proceeding, we state the following preliminary lemma. 
\begin{lem}\label{equivhess}
    Let $\langle f,g\rangle$ be a pencil of cubics and suppose that $H(g)$ is defined, 
    distinct from $g$ and $H(g)\in\langle f,g\rangle$. Then, for any $C\in SL(3)$ such 
    that $H(C\cdot g) \in \langle f,C\cdot g\rangle$, we have 
    \begin{equation*}
        C\cdot \langle f,g\rangle =\langle f,C\cdot g\rangle
        \end{equation*}
\end{lem}
\begin{proof}
    It is enough to recall that $H(C\cdot g)=C\cdot H(g)$, and thus we have: 
    \begin{equation*}
        C\cdot \langle f,g\rangle =C\cdot\langle g,H(g)\rangle =\langle C\cdot g, H(C\cdot g)\rangle =\langle f,C\cdot g\rangle
    \end{equation*}
\end{proof}
\begin{proof}[Proof of Proposition \ref{numerofinitoorbite}]
It is enough to prove the statement for representatives of the $SL(3)$–orbits of plane cubics. 
From the classification in Table \ref{tab:pencilN}, the only orbit representatives for 
which the family of pencils contained in $N$ might not be finite are 
\[
x^3\  (\textcolor{blue}{A}),\quad x^2y\ (\textcolor{blue}{B}),\quad xy(x+y)\ (\textcolor{blue}{C}),\quad xyz\ (\textcolor{blue}{D}).
\]
We now analyze these four cases separately. \revised{As the reader may expect, some orbits will 
appear multiple times in the case-by-case analysis.
Whenever this happens, we will indicate where they have already been encountered by 
referring to the appropriate table.}
\begin{itemize}
\item[\textcolor{blue}{\textbf{A}}]We began our analysis with $x^3$. Our goal is to 
show that the pencils of the form 
\begin{equation*}
    \langle x^3,3b_1 x^2y+3b_2 x^2z+3b_3xy^2+6b_4 xyz+3b_5 xz^2\rangle
    \end{equation*}
    form only finitely many $SL(3)$-orbits, \revised{which we can explicitly list}. We observe that all 
    the cubics in the linear combination are of the form $line+conic$, \revised{where the line 
    is $\{x=0\}$}. 
    Let us distinguish several cases: 
 \begin{itemize}
    \item \textbf{Conic(smooth) + Tangent line$\{x=0\}$}\\
    This condition is satisfied when the intersection of $\{x=0\}$ and the conic 
    $\{3b_1 xy+3b_2 xz+3b_3y^2+6b_4 yz+3b_5 z^2\}$ consists of a single point, and the 
    conic does not degenerate into one or two lines. The first requirement leads to the 
    condition $b_4^2=b_3b_5$, while the second one is equivalent to requiring that the 
    Hessian is non-zero and distinct from the cubic itself. \revised{Together} these conditions are 
    precisely those we found for cubics in $H^{-1}(x^3)$ (see Table~\ref{tab:hessianspace}). 
    Moreover, all cubics of the form $conic(smooth)+tangent \ \ line$ lie in the same $SL(3)$-orbit. 
    By Lemma \ref{equivhess}, it follows that all the pencil of the form 
    \begin{equation*}
        \langle x^3, smooth\ \ conic\ \  tangent \ \ to \ \ \{x=0\}\rangle  
        \end{equation*}
    belong to the same orbit. 
    \item \textbf{Conic(smooth) + Secant line}\\
    From the classification of the $SL(3)$-orbits, we know that all cubics of this 
    type lie in the same orbit. We claim that the Hessian of such a cubic lies in 
    the pencil generated by it and by $x^3$. Then, by Lemma \ref{equivhess}, this 
    will imply that all pencils of the form 
    \begin{equation*}
        \langle x^3,smooth\ \ conic\ \ secant\ \ to \ \ \{x=0\}\rangle 
    \end{equation*}
    belong to the same orbit.\\
    Let us consider the cubic $x(x^2+yz)$, which is of this type. Its Hessian is 
    $x(-3x^2+xyz)\in \langle x^3,x(x^2+yz)\rangle $ and has the same form. If now 
    $xq$ denotes another cubic of the same form, we know that there exist a matrix 
    $C\in SL(3)$ such that $C\cdot x(x^2+yz)=xq$. Moreover, $C$ must send $x$ to 
    itself and the following diagram commutes: 
    \[
\begin{tikzcd}
     f=x(x^2+yz) \arrow[r,"\cdot C"] \arrow[d,"H"] & xq \arrow[d,"H"] \\
    H(f)\in\langle x^3,x(x^2+yz)\rangle  \arrow[r,"\cdot C"] &   C\cdot H(f)=H(xq)
\end{tikzcd}
\]
We obtain 
\begin{equation*}
    H(xq)=C\cdot H(f)\in C\cdot \langle x^3,f\rangle =\langle C\cdot x^3, C\cdot 
    f\rangle =\langle x^3,xq\rangle 
\end{equation*}
Which prove our claim. 
\item \textbf{Three non-concurrent lines (triangle)}\\
First, observe that the pencil $\langle x^3,xyz\rangle $ contains the cubic 
$x^3+xyz=x(x^2+yz)$, and therefore it coincides with the pencil 
$\langle x^3,x(x^2+yz)\rangle $, which is precisely the case previously analyzed. \\
If $x(x+\beta_1y+\beta_2z)(x+\gamma_1y+\gamma_2z)$ is another triangle of the same form, 
then  $\beta_1\gamma_2-\beta_2\gamma_1\neq 0$ and there exist $C\in SL(3)$ such 
that $C\cdot xyz = x(x+\beta_1y+\beta_2z)(x+\gamma_1y+\gamma_2z)$ and $C\cdot x=x$. 
Therefore, all pencils of the form 
\begin{equation*}
    \langle x^3, triangle(three\ \  distinct\ \  lines)\rangle 
\end{equation*}
lie in the same orbit. However, this orbit coincides with that of the pencils 
\begin{equation*}
     \langle x^3,smooth\ \ conic\ \ secant\ \ to \ \ \{x=0\}\rangle 
\end{equation*}
\item \textbf{Three concurrent lines (cone)}\\
All the cones consisting of three distinct lines lie in the same orbit. 
Moreover, if two such conics share the line $\{x=0\}$, we can map one to the other 
using an element $C\in SL(3)$ that fixes $x$. Therefore, applying Lemma \ref{equivhess}, 
we conclude that all pencils
\begin{equation*}
    \langle x^3,cone(three\ \  distinct\ \ line)\rangle 
\end{equation*}
belong to the same orbit. 
\item \textbf{line  $\{x=0\}$+ double line}\\
First, observe that the pencil $\langle x^3,xy^2\rangle $ contains the cubic 
$x^3+xy^2=x(x^2+y^2)$ that is a cone with three distinct lines. Thus, this 
pencil coincides with $\langle x^3,x(x^2+y^2)\rangle $ and it is in the previous orbit.\\
All cubics of the type $xr^2$ lie in the same orbit. In particular, since we 
must have $C\cdot xr^2=xs^2$ with $r,s \neq x$, it follows that $C$ must fix the line $\{x=0\}$. 
Therefore, by applying again Lemma \ref{equivhess}, we conclude that all pencils of the form 
\begin{equation*}
    \langle x^3,\{x=0\}+ double \ \ line\rangle 
\end{equation*}
lie in the same orbit. Moreover, based on the initial observation, this orbit coincides with 
\begin{equation*}
    \langle x^3,cone(three\ \ distinct\ \ lines)\rangle 
\end{equation*}
\item \textbf{double line  $\{x=0\}$+ line}\\
All this kind of conics lie in the same $SL(3)$-orbit. Moreover, if $C\cdot x^2r=x^2s$ 
then $C$ must fix $x$, and thus we have
\begin{equation*}
    \langle x^3, \{x=0\}^2+line\rangle 
\end{equation*}
lie in the same orbit.

\end{itemize}

In summary, we have shown that all pencils containing $x^3$ and lying in $N$ fall into 
four distinct orbits. We list in the following table a representative for each orbit, 
together with the dimension of the orbit and the rank of the Jacobian matrix $210\times 45$ of $N$. 
\renewcommand{\arraystretch}{1.72}
\begin{table}[H]
\begin{adjustwidth}{-0.8cm}{-0.8cm}

\centering
    \begin{tabular}{c|c|c|c}
    \textbf{Description}&\textbf{Representative} $\langle f,g\rangle $ & \textbf{dim($O(\langle f,g\rangle )$)}&\textbf{Rank $J$}\\
    \hline
       $\langle x^3,Three\ non{-}concurrent\ lines\rangle =$&$\langle x^3,xyz\rangle $&$6$&$36$\\
       $\langle x^3,Conic+ secant\ \ line\rangle$&&&\\ \hline
       $\langle x^3,Conic+tangent\ \  line\rangle $&$\langle x^3,x(y^2+xz)\rangle $&$5$& $36$\\ \hline
       $\langle x^3,x\cdot double\ \ line\rangle =$& $\langle x^3,xy^2\rangle $&$4$& $36 $\\ 
       $\langle x^3,Three\ concurrent\ lines\rangle $& & & \\ \hline
       $\langle x^3,x^2\cdot line\rangle $&$\langle x^3,x^2y\rangle $&$3$& $35$\\ \hline
    \end{tabular}
    \caption{Pencils of cubics contained in $N$ that are of the form 
    $\langle x^3,f\rangle $.}
    \label{tabellax^3}
    \end{adjustwidth}
\end{table}
Note that $\langle x^3,xyz\rangle $ and $\langle x^3,x(y^2+xz)\rangle $
are two orbits of the form $\langle f,H(f)\rangle $, and they were already listed in Table \ref{tab:HF-list}
\item[\textcolor{blue}{\textbf{B}}]Let us now consider the case of $x^2y$. The pencils we 
need to analyze are of the form 
\begin{equation*}
    \langle x^2y, b_0x^3+3b_2x^2z+b_6y^3\rangle 
\end{equation*}
However, the space of cubics defined by such the cubics $b_0x^3++3b_2x^2z+b_6y^3$ with 
$b_2,b_6\neq 0$ coincides with $H^{-1}(x^2y)$. All these pencils are of the form 
$\langle f,H(f)\rangle $, and therefore, by the Remark \ref{pencil<f,H(f)>}, they lie 
in the same orbit as $\langle y^2z-x^3,xy^2\rangle $, which was already listed in 
Table~\ref{tab:HF-list}. For the remaining cases, with $b_2=0$ or $b_6=0$, it is 
straightforward to verify that there exists an element of $SL(3)$ preserving $x^2y$ which 
maps $x^3+3b_2x^2z$ to $x^2z$ and $x^3+b_6y^3$ to $x^3+y^3$. We find five orbits:  
\renewcommand{\arraystretch}{1.72}
\begin{table}[H]

\centering
    \begin{tabular}{c|c|c}
    \textbf{Representative} $\langle f,g\rangle $ & \textbf{dim($O(\langle f,g\rangle )$)}&\textbf{Rank $J$}\\ \hline
    $\langle x^2y,y^3+x^2z\rangle $&$ 6$&$36$\\ \hline
    $\langle x^2y,x^3+y^3\rangle  $ &$5$ &$36$ \\ \hline 
    $\langle x^2y,x^2z\rangle ,\ \langle x^2y,x^3+x^2z\rangle $&$4$& $35$\\ \hline
    
    $\langle x^2y,y^3\rangle$&$4$& $36$\\ \hline 
    $\langle x^2y,x^3\rangle $&$3$& $35$\\ \hline 
       \end{tabular} 
       \caption{Pencils of cubics contained in $N$ that are of the form 
       $\langle x^2y,f\rangle $.}
    \label{tabellax^2y}

\end{table}
It can be observed that the last two orbits 
in the table coincide with two that had already been listed in Table~\ref{tabellax^3}. 
Thus, the only new orbits with respect to the previously classified ones are those of 
$\langle x^2y,x^3+y^3\rangle $ and $\langle x^2y,x^2z\rangle $.

\item[\textcolor{blue}{\textbf{C}}]Let us now consider the case of $xy(x+y)$. 
The pencils we need to analyze are of the form: 
\begin{equation*}
    \langle xy(x+y),b_0x^3+3b_2(x^2z+xyz+y^2z)+b_6y^3\rangle 
\end{equation*}
We denote by $g$ the linear combination $b_0x^3+3b_2(x^2z+xyz+y^2z)+b_6y^3$. 
A straightforward computation of the Hessian shows that $H(g)\in\langle xy(x+y),g\rangle $ 
for every $g$ whose Hessian is defined. Therefore, for such $g$, the orbit coincides with 
one of the form $\langle f,H(f)\rangle $ already listed in Table~\ref{tab:HF-list}, 
namely the one corresponding to $\langle y^2z-x^3-x^2z,3xy^2-x^2z+y^2z\rangle $. 
The remaining cases are pencil of the form $\langle xy(x+y),x^3+b_6y^3\rangle $. 
For these we observe that the pencil contains a $triple\ \  line$ only when $b_6=0$ or $b_6=1$; 
in all other cases, it instead contains a $double\   line + line$. We can thus reduce 
the first case to the orbit of $\langle x^3,xy^2\rangle$ and the second one to the orbit 
of $\langle x^2y,x^3+y^3\rangle $, both of which have already been analyzed. 
We find three orbits of pencils. 
\renewcommand{\arraystretch}{1.72}
\begin{table}[H]
\begin{adjustwidth}{-1.5cm}{-1.5cm}

\centering
    \begin{tabular}{c|c|c}
    \textbf{Representative} $\langle f,g\rangle $ & \textbf{dim($O(\langle f,g\rangle )$)}&\textbf{Rank $J$}\\
    \hline
    $\langle xy(x+y),z(x^2+xy+y^2)\rangle ,\ \langle xy(x+y),x^3+z(x^2+xy+y^2)\rangle , $ &$7$ &$36$ \\ 
    $\langle xy(x+y),x^3+y^3+z(x^2+xy+y^2)\rangle $& & \\ \hline
    $\langle xy(x+y),x^3+y^3\rangle ,\ \langle xy(x+y),x^3\rangle $&$4$& $36$\\ \hline
       $\langle xy(x+y),x^3+b_6y^3\rangle\quad b_6\neq 1,0 $&$5$&$36$\\ \hline
       \end{tabular}
    \caption{Pencils of cubics contained in $N$ that are of the form $\langle xy(x+y),f\rangle $.}
    \label{tabellaxy(x+y)}
    \end{adjustwidth}
\end{table}
\textcolor{black}{These three orbits coincide with three that were already found previously in table~\ref{tab:HF-list},
\ref{tabellax^3} and \ref{tabellax^2y} respectively.}

\item [\textcolor{blue}{\textbf{D}}]The last case we need to study is that of $xyz$. 
We consider pencils of the form 
\begin{equation*}
    \langle xyz,b_0x^3+b_6y^3+b_9z^3\rangle 
\end{equation*}
and observe that we can act with $SL(3)$ while fixing $xyz$. Therefore, there are three orbits. 
\renewcommand{\arraystretch}{1.5}
\begin{table}[H]

\centering
    \begin{tabular}{c|c|c}
    \textbf{Representative} $\langle f,g\rangle $ & \textbf{dim($O(\langle f,g\rangle )$)}&\textbf{Rank $J$}\\
    \hline
    $\langle xyz,x^3+y^3+z^3\rangle $&$8$& $36$\\ \hline $\langle xyz,x^3+y^3\rangle $&$6$& $36 $\\ \hline 
    $\langle xyz,x^3\rangle $ &$6$ &$36$ \\ \hline
       \end{tabular}
    \caption{Pencils of cubics contained in $N$ that are of the form $\langle xyz,f\rangle $.}
    \label{tabellaxyz}
\end{table}
\end{itemize}
\textcolor{black}{These three orbits are all of the form $\langle f,H(f)\rangle $, and therefore coincide 
    with three of those in Table~\ref{tab:orbitsN}. In particular, the second one coincides 
    with the orbit of $\langle y^2z-x^3,xy^2\rangle $.}\\

To summarize what we have done so far: the analysis carried out up to this point has 
shown that $N$ contains, in addition to the orbit of the pencil $\langle x^3+y^3+z^3,xyz\rangle$, 
which is, in particular, the only one of dimension $8$, eight more orbits. Among these, 
four are of the form $\langle f,H(f)\rangle $, while the remaining four have dimensions 
$5,4,4,$ and $3$. 
\renewcommand{\arraystretch}{1.8}
\begin{table}[H]
\begin{adjustwidth}{-2.1cm}{-2.1cm}

\centering
    \begin{tabular}{c|c}
    \textbf{Representative} $\langle f,g\rangle $ & \textbf{dim($O(\langle f,g\rangle) $)} \\
    \hline
    $\langle x^3+y^3+z^3,xyz\rangle$& $8$\\ \hline
       $\langle y^2z-x^3-x^2z,3xy^2-x^2z+y^2z\rangle $  & $7$ \\ \hline
       $\langle y^2z-x^3,xy^2\rangle $ &$6$\\ \hline
       $\langle x^3,xyz\rangle $&$6$\\ \hline
       $\langle x(y^2+xz),x^3\rangle $&$5$\\ \hline 
       $\langle x^2y,x^3+y^3\rangle $& $5$\\ \hline 
       $\langle x^3,xy^2\rangle $& $4$\\ \hline
       $\langle x^2y,x^2z\rangle $& $4$\\ \hline
       $\langle x^3,x^2y\rangle $& $3$\\ \hline
       
    \end{tabular}
    \caption{Classification of orbits in $N$}
    \label{tab:orbitsN}
    \end{adjustwidth}
\end{table} 
\revised{The computation of the rank of the Jacobian matrix
shows that the rank equals the expected codimension 
in all cases except for 
$O(\langle x^2y,x^2z\rangle)$ and $O(\langle x^3,x^2y\rangle)$, 
where it drops. 
Hence these two orbits constitute the singular locus of the variety.}
\end{proof}
 
\revised{To conclude this section, we show that $N$ is reduced. 
To this end, we study the scheme locally in neighborhoods of the singular points 
in order to detect possible embedded components 
(see Proposition~4.9 in \cite{atiyahmacdonald}). 
This allows us to work in local charts of the Grassmannian, and hence in an 
affine space with $16$ variables. We then use computational methods on these 
local schemes.\\
To obtain a local chart in a neighborhood of $\langle x^3, x^2y \rangle$, we 
consider a $2\times 10$ matrix of the form
\[
\begin{bmatrix}
1 & 0 & a_2 & \dots & a_9 \\
0 & 1 & b_2 & \dots & b_9 
\end{bmatrix}.
\]
The Plücker coordinates are given by the $2\times2$ minors of this matrix. 
Substituting them into the equations defining $N$, we obtain an ideal in the 
$16$ local coordinates $a_2,\dots,a_9,b_2,\dots,b_9$. Using \texttt{Macaulay2}, 
we verify that this ideal is radical, and hence the corresponding scheme is reduced.\\
Repeating the same argument in a neighborhood of $\langle x^2y, x^2z \rangle$, 
and thus considering the matrix
\[
\begin{bmatrix}
a_0 & 1 & 0 & a_3 & \dots & a_9 \\
b_0 & 0 & 1 & b_3 & \dots & b_9 
\end{bmatrix},
\]
we again obtain a reduced scheme.}

\subsection{Proof of Main Result}
All the arguments developed so far lead to the proof of the main result of this article, 
namely
\mainthm*

\begin{proof}
\revised{Since $SL(3)$ is irreducible and the image of an irreducible variety under 
a morphism is irreducible, the orbit is irreducible, and hence so is its closure (\cite{Linearalggroups}).}\\
For the remainder, it sufficies to prove that the two varieties $N$ (Definition~\ref{defN}) 
and $H_8$ (Definition \ref{defXX})
coincide. \\
Both $N$ and $H_8$ are $8$-dimensional varieties contained in $G(1,9)$, with $H_8 \subset N$. 
From Proposition \ref{multigradoScubiche} and Table \ref{tab:multigradoN}, it follows 
that these two varieties have the same multidegree with respect to the 8-dimensional 
Schubert cycles in $G(1,9)$, namely $(1,3,9,12,6)$. By the Basis Theorem 
(see \cite{Schubertcalculus}), it follows that these varieties coincide in dimension $8$. 

To complete the proof of the theorem, it will therefore be enough to show that 
all orbits of pencils of cubics contained in $N$ also lie in $H_8$, that is, 
in the closure of $O(\langle x^3+y^3+z^3,xyz\rangle)$. To do so, we explicitly 
construct degenerations which, as $\epsilon \rightarrow 0$, tend to pencils 
lying in the smaller orbits.

The degenerations of the orbits of the form $\langle f, H(f)\rangle$ 
are obtained by considering families of smooth cubics degenerating to $f$, 
together with their corresponding Hessians. 
The same argument applies to the orbits 
$\langle x^2y, x^3+y^3\rangle$ and $\langle x^3, xy^2\rangle$.
In the following table, the left-hand column contains a representative 
for each orbit, while the right-hand column displays, for each such orbit, 
a family of pencils whose limit, as $\varepsilon \to 0$, is the corresponding pencil on the left.
For each family appearing in the right-hand column, one checks 
(either by hand or using Macaulay2) that it is of the form 
$\langle f(\varepsilon), H(f(\varepsilon))\rangle$, 
with $f(\varepsilon)$ smooth for $\varepsilon \neq 0$.
\renewcommand{\arraystretch}{2}
\begin{table}[H]
\begin{adjustwidth}{-2.1cm}{-2.1cm}
\centering
    \begin{tabular}{c|c}
    \textbf{Representative} $\langle f,g\rangle $ & \textbf{Degeneration} \\
    \hline
       $\langle y^2z-x^3-x^2z,3xy^2-x^2z+y^2z\rangle $  & $\langle y^2z-x^3-x^2z+\epsilon z^3,\ 3xy^2-x^2z+y^2z+ \epsilon(-9xz^2-3z^3)\rangle $ \\ \hline
       $\langle y^2z-x^3,xy^2\rangle $ &$\langle y^2z-x^3-\epsilon z^3,\ xy^2+3\epsilon xz^2\rangle $\\ \hline
       $\langle x^3,xyz\rangle $&$\langle x(x^2+yz)+\epsilon (y^3+z^3),\ -6x^3+2xyz+\epsilon(216xyz\epsilon-6y^3-6z^3)\rangle $\\ \hline
       $\langle x(y^2+xz),x^3\rangle $&$\langle x(y^2+xz)+\epsilon z^3,\ x^3-\epsilon(-3y^2z+3xz^2)\rangle $\\ \hline 
       $\langle x^2y,x^3+y^3\rangle $& $\langle x^3+2y^3+(x+\epsilon z)^3, xy(x+\epsilon z)\rangle $\\ \hline 
       $\langle x^3,xy^2\rangle $& $\langle x^3+y^3+(\epsilon z-y)^3,xy(\epsilon z-y)\rangle $\\ \hline
    \end{tabular}
    \caption{Degeneration families for all the orbits in $N$
    of the form $\langle f,H(f)\rangle$ and for $O(\langle x^2y,x^3+y^3\rangle )$, 
    $O(\langle x^3,xy^2\rangle)$}
    \label{tab:degenerations}
\end{adjustwidth}
\end{table}

The degenerations of $\langle x^2y, x^2z\rangle$ and 
$\langle x^3, x^2y\rangle$ are obtained by a slightly different argument.
For $\langle x^2y, x^2z\rangle$, we consider a family of pencils 
in $\langle xy(x+y), z(x^2+xy+y^2)\rangle$, depending on $\varepsilon$, 
which degenerates to $\langle x^2y, x^2z\rangle$ as $\varepsilon \to 0$.
One observes that 
$\langle xy(x+y), z(x^2+xy+y^2)\rangle$ lies in the same orbit as 
$\langle y^2z - x^3 - x^2z, 3xy^2 - x^2z + y^2z\rangle$; 
this follows from the analysis in point \textcolor{blue}{C}. 
In particular, we consider
\[
\langle xy(x+\varepsilon y), 
z(x^2+\varepsilon xy+\varepsilon^2 y^2)\rangle  
\in O(\langle xy(x+y), z(x^2+xy+y^2)\rangle)
= O(\langle y^2z-x^3-x^2z, 3xy^2-x^2z+y^2z\rangle).
\]
This shows that $\langle x^2y, x^2z\rangle$ belongs to the closure of the Hesse pencil $H_8$, 
since we have already observed that the orbit 
$O(\langle y^2z-x^3-x^2z, 3xy^2-x^2z+y^2z\rangle)$ 
is contained in $H_8$, and therefore any pencil in its closure is contained in $H_8$ as well.
Finally, the degeneration of $\langle x^3, x^2y\rangle$ is obtained from 
$\langle x^3, xy(x+\varepsilon z)\rangle \in O(\langle x^3, xyz\rangle)$.
This shows that $\langle x^3, x^2y\rangle$ lies in the closure of 
$O(\langle x^3, xyz\rangle)$, which we have already proved is contained in $H_8$. 
By transitivity of inclusion, it follows that 
$\langle x^3, x^2y\rangle$ is also contained in $H_8$.

For completeness, we report in the following table the degeneration families
of the last two orbits considered.
\renewcommand{\arraystretch}{2}
\begin{table}[H]
\begin{adjustwidth}{-2.1cm}{-2.1cm}
\centering
    \begin{tabular}{c|c}
    \textbf{Representative} $\langle f,g\rangle $ & \textbf{Degeneration} \\
    \hline
       $\langle x^2y,x^2z\rangle $& $\langle xy(x+\epsilon y), z(x^2+\epsilon xy+\epsilon^2 y^2)\rangle $\\ \hline
       $\langle x^3,x^2y\rangle $& $\langle x^3,xy(x+\epsilon z)\rangle $\\ \hline
    \end{tabular}
    \caption{Degeneration families for $O(\langle x^2y,x^2z\rangle)$ and $O(\langle x^3,x^2y\rangle)$}
    \label{tab:deg-other}
\end{adjustwidth}
\end{table}

Thus, we have proved that all the orbits contained in $N$ lie in the closure of 
$O(\langle x^3+y^3+z^3, xyz\rangle)$, and therefore we conclude that $H_8 = N$.\\
\revised{The statement about the singular locus follows from 
Proposition~\ref{numerofinitoorbite}, where we show that the singular locus 
of $N$ is given by the two claimed orbits, together with the fact that $N$ is 
reduced, as shown at the end of Section~\ref{sectionR}.}
This completes the proof of Theorem~\ref{teoS=N}.
\end{proof}

\begin{oss}
\revised{The canonical class $K_{H_8}$ remains unknown. 
Since $H_8$ is not a complete intersection, the adjunction formula 
cannot be applied, unlike in the case of binary quartics 
(see Remark~\ref{canonicalbundlequartiche}). 
A computational approach was also attempted, but the complexity of the 
equations involved did not allow us to obtain a result.}
\end{oss}

\section*{Acknowledgements}
The author would like to thank Giorgio Ottaviani at the University of Florence for 
suggesting the problem and for his guidance during the development of this work.

\section*{Statements and Declarations}
\textbf{Competing Interests.} The author declares that she has no competing interests.

\end{document}